\documentclass[final,a4paper,11pt]{article}
\usepackage{graphicx}
\usepackage{mathptmx}
\usepackage{amssymb,amsmath, amsfonts,amsthm}
\usepackage{a4wide}
\usepackage{color}

\usepackage{hyperref}
\usepackage{tikz}
\usepackage{comment}
\usepackage{cases}
\usepackage{array}
\usepackage{tabularx}
\usepackage{ragged2e}
\usepackage{comment}

\newcommand{\email}[1]{{\tt #1}}
\newcommand{\R}{\mathbb{R}}
\newcommand{\oR}{\overline{\R}}

\newcommand{\norm}[1]{\|#1\|}

\newcommand{\mv}{\,\mid\,}
\newcommand{\bmv}{\,\big\vert\,}

\newcommand{\B}{{\cal B}}
\newcommand{\oB}{\bar{\cal B}}

\newcommand{\I}{{\cal I}}

\newcommand{\M}{{\cal M}}

\newcommand{\T}{{\cal T}}

\newcommand{\Sp}{{\mathcal S}}

\newcommand{\Z}{{\cal Z}}

\newcommand{\setto}[1]{\mathop{\rightarrow}\limits^#1}
\newcommand{\attto}[1]{\mathop{\rightarrow}\limits_#1}
\newcommand{\longsetto}[1]{\mathop{\longrightarrow}\limits^{#1}}

\newcommand{\attconv}[1]{\mathop{\longrightarrow}\limits_{#1}^{\gph \partial #1}}
\newcommand{\skalp}[1]{\langle #1\rangle}

\newcommand{\xb}{\bar x}
\newcommand{\yb}{\bar y}

\newcommand{\AT}[2]{{\textstyle{#1\atop#2}}}
\newcommand{\xba}{{\bar x^\ast}}

\newcommand{\OO}{{\cal O}}
\newcommand{\argmin}{\mathop{\rm arg\,min}}
\newcommand{\argmax}{\mathop{\rm arg\,max}\limits}

\newcommand{\inn}{{\rm int\,}}

\newcommand{\co}{{\rm conv\,}}
\newcommand{\gph}{\mathrm{gph}\,}
\newcommand{\dom}{\mathrm{dom}\,}
\newcommand{\epi}{\mathrm{epi}\,}
\newcommand{\tto}{\rightrightarrows}
\newcommand{\Limsup}{\mathop{{\rm Lim}\,{\rm sup}}}

\newcommand{\rge}{{\rm rge\;}}

\newcommand{\Ta}{\mathcal{T}^f_\epsilon}
\newcommand{\varco}{{\rm var\; co\,}}
\newcommand{\tilt}{{\rm tilt\,}}

\newlength{\myAlgBox}
\newtheorem{theorem}{Theorem}[section]
\newtheorem{proposition}[theorem]{Proposition}
\newtheorem{remark}[theorem]{Remark}
\newtheorem{lemma}[theorem]{Lemma}
\newtheorem{corollary}[theorem]{Corollary}
\newtheorem{definition}[theorem]{Definition}

\title{On second-order variational analysis of variational convexity of prox-regular functions}
\author{Helmut Gfrerer\thanks{Institute of Computational Mathematics, Johannes Kepler University
Linz  and Johann Radon Institute for Computational and Applied Mathematics (RICAM), Austrian Academy of Sciences, A-4040 Linz, Austria; \email{helmut.gfrerer@ricam.oeaw.ac.at}; ORCID 0000-0001-6860-102X}
}

\date{}

\begin{document}
\maketitle

{\footnotesize
\noindent{\bf Abstract.} Variational convexity, together with ist strong counterpart, of extended-real-valued functions has been recently introduced by Rockafellar. In this paper we present second-order characterizations of these properties, i.e., conditions using first-order generalized derivatives of the subgradient mapping. Up to now, such characterizations are only known under the assumptions of prox-regularity and subdifferential continuity and in this paper we discard the latter. To this aim we slightly modify the definitions of the generalized derivatives to be compatible with the $f$-attentive convergence appearing in the definition of subgradients. We formulate our results in terms of both coderivatives and subspace containing derivatives. We also give formulas for the exact bound of variational convexity  and study relations between variational strong convexity, tilt-stable local minimizers and strong metric regularity of some truncation of the subgradient mapping.
}

\section{Introduction}

It is well-known that the subgradient mapping $\partial f$  of a lower semicontinuous (lsc) function $f:\R^n\to\oR:=(-\infty,\infty]$ is maximally monotone if and only if the function is convex. If one replaces maximal monotonicity of $\partial f$ by local monotonicity around some reference pair then one arrives at the notion of {\em variational convexity} of $f$, which, loosely speaking, states that some localization of the subgradient mapping $\partial f$ is indistinguishable from that of a convex function. This is a more subtle property than local convexity of $f$. Variational strong convexity can be defined in this scheme with replacing convexity by strong convexity and monotonicity of the subgradient mapping by strong monotonicity.  Variational strong convexity (without naming it) goes back to the characterization of {\em tilt-stable} local minimizers by Poliquin and Rockafellar \cite{PolRo98}, but the definitions of variational (strong) convexity were only introduced more than twenty years later by Rockafellar \cite{Ro19}.

Variational (strong) convexity is an important tool in the analysis of nonconvex optimization problems. E.g., in the presence of variational convexity, the first-order optimality condition $0\in\partial f(\xb)$ is also sufficient for $\xb$ being  a local minimizer, whereas variational strong convexity also ensures that $\xb$ is a tilt-stable local minimizer \cite{Ro19}. The reader can find other applications  to the proximal point and Augmented Lagrangian methods in the recent papers \cite{Ro19, Ro23a, Ro23b}.

Another important concept in variational analysis is {\em prox-regularity} as introduced by Poliquin and Rockafellar \cite{PolRo96}. Prox-regularity has  numerous applications  in the stability analysis of solutions of nonconvex optimization problems and, more generally, in second-order calculus. Its close relation with variational convexity was pointed out by Rockafellar \cite{Ro24}.

As already mentioned, variational (strong) convexity can be characterized by the local (strong) monotonicity of the subgradient mapping. There are other characterizations available in the literature, e.g., by the uniform quadratic growth condition \cite{Ro19,Ro24} or by local (strong) convexity of the Moreau envelope \cite{KhMoPh23a}. The aim of this paper is to provide second-order characterizations of (strong) variational convexity of prox-regular functions, where second-order analysis is understood as first-order analysis of the subgradient mapping. Another type of second-order analysis by using second-order expansions of the function, is out of the scope of this paper.
To the best of our knowledge, such second-order characterizations in terms of regular/limiting coderivatives of $\partial f$ are only known under the additional assumption that $f$ is subdifferentially continuous \cite{KhMoPh23a}. We also refer the reader to the very recent book of Mordukhovich \cite{Mo24}, where results for the infinite-dimensional case can be found as well.

The reason why the classical second-order constructions do not work in the absence of subdifferential continuity, can be  simply explained by the fact that for variational convexity we have to work with $f$-attentive localizations of the subgradient mapping whereas the usual limiting coderivative of $\partial f$ does not take $f$-attentive convergence into account. Two possible remedies appear to be manifest: We can either work with generalized derivatives of $f$-attentive localizations  of $\partial f$ or we modify the definition of the generalized derivative so that $f$-attentive convergence is incorporated. Although we will show that both approaches are equivalent, the former  has the disadvantage that usually we do not know in advance the $f$-attentive neighborhood which we have to work with,  whereas the latter results in purely point-based objects. We refer the reader to the very recent paper \cite{KhKhMoPh23c} where a neighborhood-based characterization of variational convexity was given in terms of the regular/limiting coderivative of some $f$-attentive localization of $\partial f$. 

Another goal of this paper is to point out that for characterizing variational (strong) convexity we are not restricted to the use of $f$-attentive limiting coderivatives but we can equivalently use so-called  {\em subspace containing derivatives} (SC derivatives) as introduced in \cite{GfrOut22}. The SC derivative of a set-valued mapping is a generalized  derivative whose elements are subspaces and can be considered as a generalization of the B-differential for single-valued mappings to set-valued ones. Indeed, the elements of the B-differential are linear mappings, i.e., mappings whose graph is a linear subspace, but the subspaces contained in an SC derivative are not necessarily graphs of linear mappings. The subspaces belonging to the so-called adjoint SC derivative are always contained in the graph of the limiting coderivative and therefore the adjoint SC derivative can be considered as a kind of skeleton for the limiting coderivative. Moreover, the computation of the SC derivative seems to be simpler than the one of the limiting coderivative. SC derivatives have been successfully applied in various fields, like the development of semismooth$^*$ Newton methods \cite{GfrOut22, GfrMaOutVal23, GfrOutVal23, Gfr24}, the characterization of tilt-stable local minimizers of prox-regular and subdifferentially continuous functions \cite{GfrOut22}, the isolated calmness property on a neighborhood \cite{GfrOut22,GfrOut23} and computing semismooth derivatives of implicitly defined mappings \cite{GfrOut24a}.

The paper is organized in the following way. Section 2 contains all the known notions and results of variational analysis and generalized differentiation used in this paper. In Section 3 we introduce the new notion of $f$-attentive generalized derivatives of the subdifferential mapping $\partial f$, provide some elementary calculus and present their relation with generalized derivatives of $f$-attentive localizations of $\partial f$ under  prox-regularity. Section 4 provides various neighborhood-based second-order characterizations of variational convexity, whereas in Section 5 we present purely point-based characterizations of variational strong convexity and tilt-stability  and we establish equivalence between these properties and strong meric regularity of a truncation of the subgradient mapping. Moreover, formulas for the exact bound of the level of variational convexity and tilt-stability are given. The sole assumption we use is prox-regularity, subdifferential continuity is not supposed.

Given $x\in\R^n$ and $\epsilon>0$, we denote by $\B_\epsilon(x)\ (\oB_\epsilon(x))$ the open (closed) ball around $x$ with radius $\epsilon$.  For a function $f:\R^n\to\oR$ we define the level sets $L_f(\rho):= \{x\mv f(x)<\rho\}$, $\rho\in\R$. The graph of a set-valued mapping $F:\R^n\tto\R^m$ is denoted by $\gph F:=\{(x,y)\mv y\in F(x)\}$. Finally, for a nonsingular $n\times n$ matrix $A$ we define $A^{-T}:=(A^T)^{-1}$.

\section{Preliminaries from variational analysis}

Given an extended real-valued function $f:\R^n\to\oR$ and  a point $\xb\in \dom  f:=\{x\in\R^n\mv  f(x)<\infty\}$, the {\em regular subdifferential} of $f$ at $\xb$ is given by
\[\widehat\partial  f(\xb):=\Big\{x^*\in\R^n\mv\liminf_{x\to\xb}\frac{ f(x)- f(\xb)-\skalp{x^*,x-\xb}}{\norm{x-\xb}}\geq 0\Big\},\]
while the {\em (limiting) subdifferential} is defined by
\[\partial  f(\xb):=\{x^*\mv \exists x_k\attto{ f} \xb, x_k^* \to x^* \mbox{ with }x_k^*\in\widehat \partial  f(x_k)\ \forall k\},\]
where $x_k\attto{ f} \xb$ denotes {\em $f$-attentive} convergence, i.e.
\[x_k\attto{ f} \xb \ :\Longleftrightarrow\ x_k\to \xb \mbox{ and }  f(x_k)\to  f(\xb).\]

\begin{definition}\label{DefProxReg} An lsc function $f:\R^n\to\oR$ is called {\em prox-regular} at $\xb$ for $\xba$
 if $f$ is finite  at $\xb$ with $\xba\in\partial f(\xb)$ and  there
exist $\epsilon> 0$ and $r\geq 0$ such that
\[ f(x')\geq  f(x)+\skalp{x^*,x'-x}-\frac r2 \norm{x'-x}^2\]
 whenever $x'\in \B_\epsilon(\xb)$ and $(x,x^*)\in \gph \partial_\rho  f\cap (\B_\epsilon(\xb)\times \B_\epsilon(\xba))$ with $\rho:=f(\xb)+\epsilon$, where the mapping $\partial_\rho f:\R^n\tto\R^n$ given by
\[\gph \partial_\rho f:=\{(x,x^*)\in \gph\partial f\mv f(x)<\rho\}=\gph \partial f\cap (L_f(\rho)\times\R^n)\]
specifies an {\em $f$-truncation} of $\partial f$.
\end{definition}
Obviously, if the requirements of Definition \ref{DefProxReg} are fulfilled for some $r\geq 0$ and $\epsilon>0$, they are also fulfilled for any $r'\geq r$ and every $0<\epsilon'\leq\epsilon$. Further, it is easy to see that  $f$ is prox-regular at $\xb$ for $\xba\in\partial f(\xb)$ with parameter $r\geq 0$ if and only if there exists an open convex neighborhood $U\times V$ of $(\xb,\xba)$ along with some $\rho>f(\xb)$ such that
\[ f(x')\geq  f(x)+\skalp{x^*,x'-x}-\frac r2 \norm{x'-x}^2\quad\mbox{ whenever }x'\in U,\ (x,x^*)\in\gph\partial_\rho f\cap(U\times V).\]

The requirement $f(x)< \rho$ ensures that the considered subgradients $x^*$ correspond to normals to $\epi  f$ at points $(x, f(x))$ close to $(\xb, f(\xb))$. This condition can be omitted if $f$ is {\em sub\-differen\-tially continuous} at $\xb$ for $\xba$ in the sense that for any sequence $(x_k,x_k^*)\longsetto{{\gph \partial  f}}(\xb,\xba)$ we have $\lim_{k\to\infty} f(x_k)= f(\xb)$. Subdifferential continuity is widely used in second-order variational analysis but one of the main goals of the paper is to avoid it. Recall that an lsc convex function $f$ is prox-regular and subdifferentially continuous at any $\xb\in\dom f$ for every subgradient $\xba\in\partial f(\xb)$.

If $f:\R^n\to\oR$ is prox-regular at $\xb\in\dom f$ for $\xba\in\partial f(\xb)$, it is well-known that for any twice continuously differentiable function $g:\R^n\to\R$ the function $f+g$ is prox-regular at $\xb$ for the subgradient $\xba+\nabla g(\xb)\in\partial(f+g)(\xb)$, cf. \cite[Exercise 13.35]{RoWe98}. Subdifferential continuity, if present, is likewise preserved.

Allowing the prox-regularity parameter $r$ also to be negative leads to the following definition.

\begin{definition}[cf. {\cite{Ro24}}]\label{DefQuadrGr}
  Let $s$ be a real number. We say that the lsc function $f:\R^n\to\oR$ has the $s$-level {\em uniform quadratic growth property} at $\xb\in\dom f$ for the subgradient $\xba\in\partial f(\xb)$, if there is an open convex neighborhood $U\times V$ of $(\xb,\xba)$ together with some real $\rho>f(\xb)$ such that
  \begin{equation}\label{EqUnifQuadrGr} f(x')\geq  f(x)+\skalp{x^*,x'-x}+\frac s2 \norm{x'-x}^2\quad\mbox{ whenever }x'\in U,\ (x,x^*)\in\gph\partial_\rho f\cap(U\times V).\end{equation}
\end{definition}

Recall that a mapping $S:\R^n\tto\R^n$ is {\em monotone} if
\[\skalp{y_1-y_2,x_1-x_2}\geq 0\mbox{ for all }(x_1,y_1),(x_2,y_2)\in\gph S\]
and it is called {\em $\sigma$-strongly monotone} if $S-\sigma I$ is monotone.
Further, a monotone mapping $S:\R^n\tto\R^n$ is called {\em maximally monotone} if for any monotone mapping $S':\R^n\tto\R^n$ with $\gph S\subset\gph S'$ there holds $\gph S=\gph S'$.

There exist also local versions of these notions. A mapping $S:\R^n\tto\R^n$ is called monotone relative to a set $W\subset\R^n\times\R^n$, if
\[\skalp{y_1-y_2,x_1-x_2}\geq 0\mbox{ for all }(x_1,y_1),(x_2,y_2)\in\gph S\cap W.\]
The mapping $S$ is called {\em locally monotone} around a point $(\xb,\yb)\in\gph S$ if there exists  a neighborhood $U\times V$ of $(\xb,\yb)$ such that $S$ is monotone with respect to $U\times V$. Finally, $S$ is called {\em locally maximal monotone} around $(\xb,\yb)\in\gph S$ if there exists a neighborhood $U\times V$ of $(\xb,\yb)$ such that $S$ is monotone relative to $U\times V$ and $\gph S\cap(U\times V)=\gph S'\cap(U\times V)$ for any monotone mapping $S':\R^n\tto\R^n$ satisfying $\gph S'\supset \gph S\cap (U\times V)$.

The following proposition is an immediate consequence of a result proven very recently by Khanh, Khoa, Mordukhovich and Phat \cite[Theorem 3.3]{KhKhMoPh23b}, which is also valid in the infinite dimensional case.
\begin{proposition}\label{PropLocMaxMon}
  Let $S:\R^n\tto\R^n$ be a mapping and consider $(\xb,\yb)\in \gph S$. Then the following statements are equivalent.
  \begin{enumerate}
    \item[(i)] $S$ is locally maximal monotone around $(\xb,\yb)$.
    \item[(ii)] There exists a maximally monotone mapping $\bar S:\R^n\tto\R^n$ and a neighborhood $U\times V$ of $(\xb,\yb)$ such that $\gph S\cap(U\times V)=\gph \bar S\cap(U\times V)$.
  \end{enumerate}
\end{proposition}
\begin{remark}Our definition of local maximal monotonicity corresponds to the one given in \cite{PolRo98} while the property stated in Proposition \ref{PropLocMaxMon}(ii) is related to another notion of local maximal monotonicity introduced in \cite{Pen02}.
\end{remark}
Next we define $f$-attentive localizations of the subdifferential mapping $\partial f$.
\begin{definition}[cf.{\cite{Ro24}}]\label{DefAttMono}Let $f:\R^n\to\oR$ be an lsc function and let $s\in\R$ be given. We say that $\partial f$ has at $\xb\in\dom f$ for $\xba\in\partial f(\xb)$ an {\em $f$-attentive localization that is $s$-monotone}, if there exists an open convex neighborhood $U\times V$ of $(\xb,\xba)$ and a real $\rho>f(\xb)$ such that
\[\skalp{y^*-x^*,y-x}\geq s\norm{y-x}^2\quad\mbox{whenever } (x,x^*),(y,y^*)\in \gph\partial_\rho f\cap(U\times V).\]
For $s=0$ this is plain monotonicity of $\partial f$ relative to the set $(U\cap L_f(\rho))\times V$.
\end{definition}
For local maximality properties of this $f$-attentive localizations of $\partial f$ we refer to \cite{Ro24}.

Another important concept is variational (strong) convexity which was recently introduced by Rockafellar \cite{Ro19} and extended in \cite{Ro24}. We augment this definition by introducing the {\em exact bound of variational convexity}.
\begin{definition}\label{DefVarConv}
Let $s$ be a real number. The lsc function $f:\R^n\to\oR$  will be called {\em variationally $s$-convex} at  $\xb\in \dom  f$ for $\xba\in\partial  f(\xb)$ if there is  some open convex neighborhood $U\times V$ of $(\xb,\xba)$, an lsc function $\widehat f:\R^n\to\oR$ and a real  $\rho>f(\xb)$ such that the quadratically shifted function $\widehat f-\frac s2\norm{\cdot}^2$ is convex on $U$, $\widehat f \leq  f$ on $U$  and
\[\gph\partial_\rho  f\cap (U\times V) = \gph \partial \widehat f\cap (U\times V) \quad \mbox{and $f(x)=\widehat f(x)$ at the common elements $(x,x^*)$}.\]
If this holds with $s=0$, we will call $f$ variationally convex at $\xb$ for $\xba$, whereas in case $s>0$ we call $f$ to be variationally strongly convex with modulus $s$.

The exact bound for variational  convexity of $f$  is defined as the supremum of all $s$ such that $f$ is variationally $s$-convex at $\xb$ for $\xba$ and is denoted by $\varco (f;\xb|\xba)$. We set $\varco(f;\xb|\xba) =-\infty$ if no such $s$ exist.
\end{definition}
Note that $\varco(f;\xb|\xba)$ can attain also the value $\infty$. Moreover, even if $\bar s:= \varco(f;\xb|\xba)\in\R$, $f$ does not have to be variationally $\bar s$-convex. E.g., for the function $f_1(x)=-x^4$ and $f_2(x)=0$ on $\R$ we have $\varco(f_i;0|0)=0$, $i=1,2$, but $f_1$ is not variationally convex at $0$ for $0$ whereas $f_2$ has this property there.

The following fundamental equivalence between the properties from Definitions \ref{DefQuadrGr}, \ref{DefAttMono} and \ref{DefVarConv} was very recently shown by Rockafellar \cite[Theorem 1]{Ro24}.
\begin{theorem}\label{ThFundEquiv}
  Given an lsc function $f:\R^n\to\oR$, a pair $(\xb,\xba)\in\gph\partial f$ and a real $s$, the following statements are equivalent.
  \begin{enumerate}
    \item[(i)] $f$ is variationally $s$-convex at $\xb$ for $\xba$.
    \item[(ii)] $f$ has the $s$-level uniform quadratic growth property at $\xb$ for $\xba$.
    \item[(iii)] $\partial f$ has an $f$-attentive localization at $\xb$ for $\xba$  that is $s$-monotone and $\xba\in\widehat\partial f(\xb)$.
  \end{enumerate}
\end{theorem}

\begin{remark}\label{RemVarConv}
  \begin{enumerate}
    \item  The equivalence between (i) and (ii) in Theorem \ref{ThFundEquiv} has the following consequences: If $f$ is variationally $s$-convex at $\xb$ for $\xba$ then it is prox-regular at $\xb$ for $\xba$ with parameter value $r=\max\{-s,0\}$. Conversely, if $f$ is prox-regular at $\xb$ for $\xba$ with parameter value $r\geq 0$ then $f$ is also variationally $s$-convex there with $s=-r$.
    \item Let the function $f:\R^n\to\oR$ be  variationally $s$-convex at $\xb$ for $\xba\in\partial f(\xb)$ and let  $\widehat f$, $U$, $V$ and $\rho$ be as in Definition \ref{DefVarConv}.
    Then  for any convex neighborhood $\tilde U\times\tilde V\subset U\times V$ of $(\xb,\xba)$ we also have $\gph\partial_\rho  f\cap ( U\times V) = \gph \partial \widehat f\cap (\tilde U\times \tilde V)$ and $f(x)=\widehat f(x)$ at the common elements $(x,x^*)$. Further, for every $\tilde\rho\in(f(\xb),\rho]$ there is an open convex neighborhood $\tilde U\times \tilde V\subset U\times V$ of $(\xb,\xba)$ with
    $\gph \partial_{\tilde\rho}  f\cap (\tilde U\times \tilde V) = \gph \partial \widehat f\cap (\tilde U\times \tilde V)$ and $f(x)=\widehat f(x)$ at the common elements $(x,x^*)$. To see this, note that $\widehat f$ is prox-regular and subdifferentially continuous at $\xb$ for $\xba$ due to convexity of $\widehat f-\frac s2\norm{\cdot}^2$. Hence we can find an open convex neighborhood $\tilde U\times \tilde V\subset U\times V$ of $(\xb,\xba)$ such that for every $(x,x^*)\in \gph \partial \widehat f\cap(\tilde U\times \tilde V)$ there holds $f(x)=\widehat f(x)<\tilde\rho$
    and hence
    \begin{align*}\lefteqn{\gph \partial\widehat f\cap(\tilde U\times \tilde V)= \gph \partial\widehat f\cap(U\times V)\cap(\tilde U\times \tilde V)\cap (L_f(\tilde \rho)\times\R^n)}\\
    &=\gph \partial f \cap(U\times V)\cap(L_f(\rho)\times\R^n)\cap(\tilde U\times \tilde V)\cap (L_f(\tilde\rho)\times\R^n)=\gph \partial_{\tilde\rho} f\cap (\tilde U\times \tilde V).\end{align*}
    \item In Definition \ref{DefVarConv}, we can assume that  the  function $\widehat f-\frac s2\norm{\cdot}^2$ is not only locally convex on $U$, but also convex and finite on whole $\R^n$ and therefore strictly continuous. Let us consider the case $s=0$. If $f$ is variationally convex at $\xb$ for $\xba$ then by Theorem \ref{ThFundEquiv} the relation
    \eqref{EqUnifQuadrGr} holds with $s=0$  and we may assume that the convex neighborhood $U\times V$ of $(\xb,\xba)$ is bounded. Now let
    \[\widehat f(x'):= \sup\{f(x)+\skalp{x^*,x'-x}\mv (x,x^*)\in \gph \partial_\rho f \cap (U\times V)\}.\]
    As the supremum of a collection of affine functions of $x'$, $\widehat f$ is lsc convex. Further, since $f(x)$ is bounded above by $\rho$ for $(x,x^*)\in \gph\partial_\rho f$ and $U\times V$ is bounded, $\widehat f$ is everywhere finite and thus strictly continuous. Now a close look at Rockafellar's proof \cite[pp.553--554]{Ro19} of the implication $(c)\Rightarrow(b)$ in \cite[Theorem 1]{Ro19} tells us that $\widehat f\leq f$ on $U$ and that we can find some open convex neighborhood $\tilde U\times \tilde V\subset U\times V$ of $(\xb,\xba)$ such that $\gph \partial_\rho  f\cap (\tilde U\times \tilde V) = \gph \partial \widehat f\cap (\tilde U\times \tilde V)$ and $f(x)=\widehat f(x)$ at the common elements $(x,x^*)$. Thus our assertion is verified for $s=0$ and for general $s\in\R$ the statement is a consequence of Lemma \ref{LemQuadrShift} below.
  \end{enumerate}
  \end{remark}

We now turn our attention to generalized derivatives of set-valued mappings. Given a set $\Omega\subset\R^n$ and a point $\xb\in\Omega$, the tangent cone, the regular normal cone and the limiting normal cone to $\Omega$ at $\xb$ are given by
\begin{gather*}
T_\Omega(\xb):=\limsup_{t\downarrow 0}\frac{\Omega-\xb}t=\{u\in\R^n\mv \exists t_k\downarrow 0,\ x_k\setto{\Omega}\xb:\ u=\lim_{k\to\infty}\frac{x_k-\xb}{t_k}\},\\
\widehat N_\Omega(\xb):=\big(T_\Omega(\xb)\big)^\circ,\\
N_\Omega(\xb):=\limsup_{x\setto{\Omega}\xb}\widehat N_\Omega(x).
\end{gather*}
Given a set-valued mapping $F:\R^n\tto\R^m$ and a point $(\xb,\yb)\in\gph F$, we  now define the graphical derivative $DF(\xb,\yb):\R^n\tto\R^m$, the regular coderivative $\widehat D^*F(\xb,\yb):\R^m\tto\R^n$ and the limiting coderivative $D^*F(\xb,\yb):\R^m\tto\R^n$ by
\begin{gather*}\gph DF(\xb,\yb):=T_{\gph F}(\xb,\yb),\\
\gph \widehat D^*F(\xb,\yb)=\{(y^*,x^*)\mv(x^*,-y^*)\in\widehat N_{\gph F}(\xb,\yb)\},\\
\gph D^*F(\xb,\yb)=\{(y^*,x^*)\mv(x^*,-y^*)\in N_{\gph F}(\xb,\yb)\}.
\end{gather*}
Let us now consider the  SC (subspace containing) derivatives  for $F$ as defined very recently in \cite{GfrOut22, GfrOut23}. We denote by $\Z_{nm}$ the metric space of all $n$-dimensional subspaces of $\R^n\times\R^m$ with metric
\[d_\Z(L_1,L_2):=\norm{P_{L_1}-P_{L_2}},\]
where $P_{L_i}$, $i=1,2$ denotes the orthogonal projection onto $L_i$. Since orthogonal projections are bounded, the metric space $\Z_{nm}$ is compact.

Given a subspace $L\in \Z_{nm}$ we define its adjoint subspace $L^*\in\Z_{mn}$ by
\[L^*=\{(v^*,u^*)\mv (u^*,-v^*)\in L^\perp\}.\]
Note that $(L^*)^*=L$ and $d_\Z(L_1,L_2)=d_\Z(L_1^*,L_2^*)$.
\begin{definition}Let $F:\R^n\tto\R^m$ be a mapping.
\begin{enumerate}
\item
We say that $F$ is {\em graphically smooth of dimension} $n$ at $(x,y)\in\gph F$ if $T_{\gph F}(x,y)\in \Z_{nm}$. By $\mathcal{O}_{F}$
 we denote the subset of $\gph F$, where $F$ is graphically smooth of dimension $n$.
\item The {\em SC derivative} of $F$ is the mapping $\Sp F:\R^n\times\R^m\tto\Z_{nm}$ given by
\[\Sp F(x,y)=
\{L\in\Z_{nm}\mv \exists (x_k,y_k)\longsetto{\OO_F}(x,y):\ \lim_{k\to\infty} d_\Z(T_{\gph F}(x_k,y_k),L)=0\},\]
whereas the {\em adjoint SC derivative} $\Sp^* F:\R^n\times\R^m\tto\Z_{mn}$ is defined by
\[\Sp^*F(x,y)=\{L^*\mv L\in \Sp F(x,y)\}.\]
\end{enumerate}
\end{definition}
Note that the adjoint SC derivative was defined in \cite{GfrOut23} in a different but equivalent manner. The adjoint SC derivative has the remarkable property that
\begin{equation}\label{EqInclAdjSubsp}L^*\subset \gph D^*F(\xb,\yb)\ \forall L^*\in\Sp^*F(\xb,\yb).\end{equation}
The adjoint SC derivative may be therefore considered as a kind of skeleton for the limiting coderivative. Let us now discuss the differences in computing the limiting coderivative
$D^*F(\xb,\yb)$ and the SC derivative $\Sp F(\xb,\yb)$. The limiting coderivative is based on the limiting normal cone $N_{\gph F}(\xb,\yb)$ which is the upper limit of regular normal cones. By taking the definition of the regular normal cone via the polar cone of the tangent cone, this means that we have to compute tangent cones $T_{\gph F}(x,y)$ for all $(x,y)\in\gph F$ close to $(\xb,\yb)$. In contrary, in order to compute the SC derivative $\Sp F(\xb,\yb)$, we only need the tangent cone at those points $(x,y)\in\gph F$ where $T_{\gph F}(x,y)$ is a subspace. This does not only mean that we need to compute less tangent cones, but also the computation of the required tangent cones should be simpler. Usually, when the computation of a tangent cone gets involved, it is not a subspace and can be therefore discarded.

For the subdifferential of a prox-regular and subdifferentially continuous  function we have the following result, which can be derived from \cite[Proposition 3.6, Lemma 3.4]{Gfr24} and gives a certain basis representation of subspaces $L\in\Sp(\partial f)(x,x^*)$.

\begin{proposition}\label{PropSCDProxRegSubDiff}
  Assume that the lsc function $f:\R^n\to\oR$ is prox-regular and subdifferentially continuous at $\xb\in\dom f$ for $\xba\in\partial f(\xb)$. Then there exists  an open convex neighborhood $U\times V$ of $(\xb,\xba)$  such that  for every $(x,x^*)\in\gph\partial f\cap(U\times V)$ there holds $\emptyset\not=\Sp(\partial f)(x,x^*)= \Sp^*(\partial f)(x,x^*)$. Further, for  every $L\in\Sp(\partial f)(x,x^*)$ there holds $L=L^*$ and there exist unique symmetric $n\times n$ matrices $P,W$ with the following properties:
  \begin{enumerate}
  \item[(i)] $L=\rge(P,W):=\{(Pp,Wp)\mv p\in\R^n\}$,
  \item[(ii)] $P^2=P$, i.e., $P$ represents the orthogonal projection onto some subspace of $\R^n$,
  \item[(iii)] $W(I-P)=I-P$.
  \end{enumerate}
\end{proposition}

Conditions (ii) and (iii) above ensure that the following relations hold:
\begin{gather}\label{EqPW1}(I-P)W=\big(W(I-P)\big)^T=(I-P)^T=I-P=W(I-P),\\
\label{EqPW2}PW=WP=PWP,\\
\label{EqPW3}W=WP+W(I-P)=PWP+(I-P).
\end{gather}
In what follows, we denote by $\M_{P,W}\partial f(x,x^*)$ the collection of all pairs of symmetric $n\times n$ matrices $(P,W)$ satisfying properties (ii) and (iii) of Proposition \ref{PropSCDProxRegSubDiff} and $\rge(P,W)\in \Sp(\partial f)(x,x^*)$. Hence
\[\Sp(\partial f)(x,x^*)=\{\rge(P,W)\mv (P,W)\in\M_{P,W}\partial f(x,x^*)\}.\]
Note that in case when $f$ is twice continuously differentiable then $\M_{P,W}\partial f(x,x^*)=\{(I,\nabla^2 f(x))\}$, cf. \cite{Gfr24}.

Now let us recall the definition of tilt-stability.

\begin{definition}[cf. {\cite{PolRo98}}]\label{DefTiltStab}
Given $f:\R^n\to\oR$, a point $\xb\in\dom f$  is a tilt-stable local minimizer of $f$  (with modulus $\kappa$) if there exists a number $\gamma  > 0$ such that the
mapping
\begin{equation}\label{EqM_Tilt} M_\gamma:x^*\mapsto  \argmin\bigl\{  f (x)-\skalp{x^*, x}\mv  x \in  \oB_\gamma(\xb)\bigr\}
\end{equation}
is single-valued and Lipschitz continuous (with constant $\kappa$) on some neighborhood $V$ of $0 \in  \R^n$ with $M_\gamma(0) = \{\xb\}$.

The {\em exact bound of tilt stability} of $f$ at $\xb$ is defined by
\begin{eqnarray}\label{tilt-exact}
\tilt(f,\xb):=\limsup_{\AT{x^*,y^*\to0}{x^*\not=y^*}}\frac{\norm{M_\gamma(x^*)-M_\gamma(y^*)}}{\norm{x^*-y^*}}.
\end{eqnarray}
\end{definition}

For prox-regular and subdifferentially continuous functions there exist a bunch of characterizations of tilt-stability in the literature, cf. \cite{PolRo98,DruLew13, MoNg15,ChHiNg18}.
We present here the following ones. Recall that a mapping $F:\R^n\tto\R^m$ is said to be {\em strongly metrically regular} with modulus $\kappa$ at $\xb$ for $\yb\in F(\xb)$, if $F^{-1}$ has a single-valued localization around $\yb$ for $\xb$ which is Lipschitz continuous with constant $\kappa$.
\begin{theorem}\label{ThTiltProxRegSubdiffCont}
  For an lsc  function $f:\R^n\to\oR$ having $0\in\partial f(\xb)$ and such that $f$ is both prox-regular and subdifferentially continuous at $\xb$ for $\xba=0$, the following statements are equivalent:
  \begin{enumerate}
\item[(i)] $\xb$ is a tilt-stable local minimizer of $f$.
\item[(ii)] The limiting coderivative $D^*(\partial f)(\xb,0)$ is positive definite in the sense that
\[\skalp{z^*,z}>0\quad\mbox{whenever }z^*\in D^*(\partial f)(\xb,0)(z),\ z\not=0.\]
\item[(iii)] For every pair $(P,W)\in\M_{P,W}\partial f(\xb,0)$, the matrix $W$ is nonsingular and the matrix $PW^{-1}$ is positive semidefinite.
\end{enumerate}
These statements entail  that  there is $\delta>0$ such that for all $x^*$ in some neighborhood
of $0$, the mapping $M_\gamma$ in Definition \ref{DefTiltStab} has $M_\gamma(x^*)$ as the unique $x\in (\partial f)^{-1}(x^*)$  with $\norm{x-\xb}<\delta$ and consequently $\partial f$ is strongly metrically regular at $\xb$ for $0$ with modulus equal to the modulus of tilt-stability. Further,
\[\tilt(f,\xb)=\sup\{\norm{PW^{-1}}\mv (P,W)\in\M_{P,W}\partial f(\xb,0)\} = \sup\left\{\frac {\norm{z}^2}{\skalp{z^*,z}}\mv z^*\in D^*(\partial f)(\xb,0)(z)\right\}\]
 with the convention $\frac 00:=0$. 
 \end{theorem}
\begin{proof}
  The equivalence between (i) and (ii) together with the identity $M_\gamma(x^*)=(\partial f)^{-1}(x^*)\cap \B_\delta(\xb)$ (without mentioning strong metric regularity of $\partial f$) was established in \cite[Theorem 1.3]{PolRo98}. The equivalence between (i) and (iii) follows from \cite[Theorem 7.9]{GfrOut22} by using the $(P,W)$-basis representation of subspaces $L\in \Sp(\partial f)(\xb,0)$. The formulas for ${\rm tilt\;}(f,\xb)$ can be found in \cite[Theorem 7.9]{GfrOut22} and \cite[Theorem 3.6]{MoNg15}, respectively.
\end{proof}
The following equivalence between tilt-stability and variational strong convexity goes back to \cite{PolRo98}, but without naming strong variational convexity. We thus refer here to the recent paper \cite{KhMoPh23a}.
\begin{theorem}[{cf. \cite[Proposition 2.9]{KhMoPh23a}}]\label{ThEquivTilt_VSC}
 Let the lsc function  $f:\R^n\to\R$  be  prox-regular and subdifferentially continuous at $\xb \in  \dom f$  for $0\in\partial f(\xb)$. Then the
following assertions are equivalent:
\begin{enumerate}
\item[(i)]$f$  is variationally strongly convex at $\xb$  for $0$ with modulus $\sigma  > 0$.
\item[(ii)] $\xb$ is a tilt-stable local minimizer of $f$  with modulus $\sigma^{-1}$.
\end{enumerate}
\end{theorem}

\section{$f$-attentive derivatives of the subgradient mapping}
It is easy to see why the classical generalized derivatives are not suited for first-order analysis of the subdifferential mapping in the absence of subdifferential continuity. Consider the tangent cone
\[T_{\gph \partial f}(\xb,\xba)=\left\{(u,u^*)\mv \exists t_k\downarrow 0, (x_k,x_k^*)\longsetto{\gph \partial f}(\xb,\xba): (u,u^*)=\lim_{k\to\infty}\frac{(x_k,x_k^*)-(\xb,\xba)}{t_k}\right\}.\]
Here, ordinary convergence $(x_k,x_k^*)\longsetto{\gph \partial f}(\xb,\xba)$ appears, but in the definition of the subdifferential we use $f$-attentive convergence. Apparently we should not mix the two different types of convergence and therefore we should also require that $f(x_k)\to f(\xb)$ in the absence of subdifferential continuity.

In the sequel we denote $f$-attentive convergence in the graph of the subgradient mapping $\partial f$ by
\begin{equation*}
(x_k,x_k^*)\attconv{f}(\xb,\xba)\ :\Longleftrightarrow\ (x_k,x_k^*)\longsetto{\gph \partial f}(\xb,\xba) \mbox{ and }f(x_k)\to f(\xb).
\end{equation*}

The following definitions are obtained by  simply replacing in the definitions of tangents, regular/limiting normals to $\gph \partial f$ and generalized derivatives of $\partial f$ the ordinary convergence $(x,x^*)\longsetto{\gph \partial f} (\xb,\xba)$ by $(x,x^*)\attconv{f}(\xb,\xba)$.

\begin{definition}\label{DefAttDer}
  Let $f:\R^n\to\oR$ be a function and $(\xb,\xba)\in \gph \partial f$.
  \begin{enumerate}
    \item The {\em $f$-attentive tangent cone} to $\gph \partial f$ at $(\xb,\xba)$ is given by
    \begin{equation}\label{EqAttTanCone}
      T^f_{\gph \partial f}(\xb,\xba):=\left\{(u,u^*)\mv \exists t_k\downarrow 0,\ (x_k,x_k^*)\attconv{f}(\xb,\xba): (u,u^*)=\lim_{k\to\infty}\frac{(x_k,x_k^*)-(\xb,\xba)}{t_k}\right\}.
    \end{equation}
    \item The {\em $f$-attentive regular normal cone} to $\gph \partial f$ at $(\xb,\xba)$ is defined as $\widehat N^f_{\gph\partial f}(\xb,\xba):=\big( T^f_{\gph\partial f}(\xb,\xba)\big)^\circ$ and the {\em $f$-attentive limiting normal cone} to $\gph \partial f$ at $(\xb,\xba)$ amounts to
        \[N^f_{\gph\partial f}(\xb,\xba):=\Limsup_{(x,x^*)\attconv{f}(\xb,\xba)} \widehat N^f_{\gph\partial f}(x,x^*).\]
    \item The {\em $f$-attentive graphical derivative $D_f(\partial f)(\xb,\xba)$}, the {\em $f$-attentive regular  coderivative $\widehat D_f^*(\partial f)(\xb,\xba)$} and the {\em $f$-attentive limiting  coderivative $D_f^*(\partial f)(\xb,\xba)$} of $\gph \partial f$ at $(\xb,\xba)$ are the set-valued mappings from $\R^n$ to the subsets of $\R^n$ given by
        \begin{gather*}
          \gph D_f(\partial f)(\xb,\xba):=T^f_{\gph \partial f}(\xb,\xba),\\
          \gph \widehat D_f^*(\partial f)(\xb,\xba) :=\{(u,u^*)\mv (u^*,-u)\in \widehat N^f_{\gph\partial f}(x,x^*)\},\\
          \gph  D_f^*(\partial f)(\xb,\xba) :=\{(u,u^*)\mv (u^*,-u)\in N^f_{\gph\partial f}(x,x^*)\}.
        \end{gather*}
    \item Let $\OO_{\partial f}^f$ denote the collection of all points $(x,x^*)\in\gph \partial f$ such that $T^f_{\gph \partial f}(x,x^*)$ is an $n$-dimensional subspace. Then the {\em $f$-attentive SC derivative} of $\partial f$ at $(\xb,\xba)$ is defined by
        \[\Sp_f(\partial f)(\xb,\xba):=\{L\in\Z_{nn}\mv \exists (x_k,x_k^*)\subset \OO^f_{\partial f}, (x_k,x_k^*)\attconv{f}(\xb,\xba): d_{\Z}\big(L,T^f_{\gph \partial f}(x_k,x_k^*)\big)\to0\}.\]
        Finally the {\em adjoint $f$-attentive SC derivative} of $\partial f$ at $(\xb,\xba)$ is defined by
        \[\Sp_f^*(\partial f)(\xb,\xba):=\{L^*\mv L\in \Sp_f(\partial f)(\xb,\xba)\}.\]
  \end{enumerate}
\end{definition}
If $f$ is subdifferentially continuous at $\xb$ for the subgradient $\xba\in\partial f(\xb)$ it follows immediately from the definitions that
\begin{gather*}T_{\gph \partial f}^f(\xb,\xba)= T_{\gph \partial f}(\xb,\xba),\ \widehat N_{\gph \partial f}^f(\xb,\xba)= \widehat N_{\gph \partial f}(\xb,\xba),\\
 D_f(\partial f)(\xb,\xba)=D(\partial f)(\xb,\xba),\ \widehat D^*_f(\partial f)(\xb,\xba)=\widehat D^*(\partial f)(\xb,\xba).\end{gather*}
 If there exists a neighborhood $U\times V$ of $(\xb,\xba)$ such that for every $(x,x^*)\in\gph\partial f\cap (U\times V)$ the function $f$ is subdifferentially continuous at $x$ for $x^*$ then there also holds
 \begin{gather*}
   N_{\gph \partial f}^f(\xb,\xba)= N_{\gph \partial f}(\xb,\xba),\ D^*_f(\partial f)(\xb,\xba)=\widehat D^*(\partial f)(\xb,\xba),\\
   \Sp_f(\partial f)(\xb,\xba)=\Sp(\partial f)(\xb,\xba),\ \Sp^*_f(\partial f)(\xb,\xba)=\Sp^*(\partial f)(\xb,\xba).
 \end{gather*}
For the $f$-attentive objects there hold the following elementary calculus rules.
\begin{proposition}\label{PropCalc}
Let $f:\R^n\to\oR$ be a function, let $(\xb,\xba)\in \gph \partial f$, let $g:\R^n\to\R$ be twice continuously differentiable and consider the function $h:=f+g$. Then
\begin{align*}&T_{\gph \partial h}^h(\xb,\xba+\nabla g(\xb))=\bar A T_{\gph \partial f}^f(\xb,\xba),\\
 &\widehat N_{\gph \partial h}^h(\xb,\xba+\nabla g(\xb))= \bar A^{-T}\widehat N_{\gph \partial f}^f(\xb,\xba),\ N_{\gph \partial h}^h(\xb,\xba+\nabla g(\xb))= \bar A^{-T} N_{\gph \partial f}^f(\xb,\xba),\\
 &\gph D_h(\partial h)(\xb,\xba+\nabla g(\xb))=\bar A\,\gph D_f(\partial f)(\xb,\xba)\\
 &\gph \widehat D_h^*(\partial h)(\xb,\xba+\nabla g(\xb))=\bar A\,\gph \widehat D_f^*(\partial f)(\xb,\xba),\\
 &\gph  D_h^*(\partial h)(\xb,\xba+\nabla g(\xb))=\bar A\, \gph D_f^*(\partial f)(\xb,\xba)\\
 &\Sp_h(\partial h)(\xb,\xba+\nabla g(\xb))=\{\bar A L\mv L\in\Sp_f(\partial f)(\xb,\xba)\},\\
 &\Sp_h^*(\partial h)(\xb,\xba+\nabla g(\xb))=\{\bar A L^*\mv L\in\Sp_f(\partial f)(\xb,\xba),\}
 \end{align*}
 where
 \[\bar A=\begin{pmatrix}
   I&0\\\nabla^2g(\xb)&I
 \end{pmatrix}.\]
\end{proposition}
\begin{proof}
  By elementary calculus rules, cf. \cite[Exercise 8.8]{RoWe98}, we have $\partial h(x)=\partial f(x)+\nabla g(x)$ for all $x$  where $f(x)$ is finite. Now consider a point $(x,x^*)\in\gph\partial f$. Then it is easy to see that for any sequence $(x_k,x_k^*)\in\R^n\times\R^n$ there holds $(x_k,x_k^*)\attconv{f}(x,x^*)$ if and only if $(x_k,x_k^*+\nabla g(x_k))\attconv{h}(x,x^*+\nabla g(x))$. Further, for every sequence $t_k\downarrow 0$ there holds $((x_k,x_k^*)-(x,x^*))/t_k\to(u,u^*)$ if and only if
  \[\lim_{k\to\infty}\frac{(x_k,x_k^*+\nabla g(x_k))-(x,x^*+\nabla g(x))}{t_k}=(u,u^*+\nabla^2g(x)u).\]
   Hence $T_{\gph\partial h}^h(x,x^*+\nabla g(x))=A(x)T_{\gph \partial f}^f(x,x^*)$ with $A(x):=\left(\begin{smallmatrix}I&0\\\nabla^2g(x)&I\end{smallmatrix}\right)$. The matrix $A(x)$  is apparently nonsingular and hence
   \begin{align*}\widehat N_{\gph\partial h}^h(x,x^*+\nabla g(x))&=\big(A(x)T_{\gph \partial f}^f(x,x^*)\big)^\circ=\big(A(x)\co T_{\gph \partial f}^f(x,x^*)\big)^\circ\\
   &=A(x)^{-T}\big(\co T_{\gph \partial f}^f(x,x^*)\big)^\circ=A(x)^{-T}\widehat N_{\gph \partial f}^f(x,x^*) \end{align*}
   by \cite[Corollary 16.3.2]{Ro70}. Since $\lim_{x\to\xb}A(x)$ equals  the nonsingular matrix $\bar A$, we may also infer that
   \begin{align*}N_{\gph\partial h}^h(\xb,\xba+\nabla g(\xb))&=\limsup_{(x,x^*+\nabla g(x))\attconv{h}(\xb,\xba+\nabla g(\xb))}\widehat N_{\gph\partial h}^h(x,x^*+\nabla g(x))\\
   &=\limsup_{(x,x^*)\attconv{f}(\xb,\xba)}A(x)^{-T}\widehat N_{\gph \partial f}^f(x,x^*)=\bar A^{-T}N_{\gph \partial f}^f(\xb,\xba).\end{align*}
   Taking into account the identities $\gph \widehat D^*_f(\partial f)(\xb,\xba)=S\widehat N_{\gph \partial f}^f(\xb,\xba)$, $\gph  D^*_f(\partial f)(\xb,\xba)=S N_{\gph \partial f}^f(\xb,\xba)$ and $S\bar A^{-T}S^{-1}=\bar A$, where $S:=\left(\begin{smallmatrix}0&-I\\I&0\end{smallmatrix}\right)$, we obtain that
   \[\gph \widehat D^*_h(\partial h)(\xb,\xba+\nabla g(\xb))=S\bar A^{-T}\widehat N_{\gph \partial f}^f(\xb,\xba)=S\bar A^{-T}S^{-1}\gph \widehat D^*_f(\partial f)(\xb,\xba)=\bar A\, \gph \widehat D^*_f(\partial f)(\xb,\xba)\]
   and, analogously, that $\gph D^*_h(\partial h)(\xb,\xba+\nabla g(\xb))=\bar A\, \gph D^*_f(\partial f)(\xb,\xba)$. Finally, there holds $\OO_{\partial h}^h=\{(x,x^*+\nabla g(x))\mv (x,x^*)\in\OO_{\partial f}^f\}$ and $\Sp_h(\partial h)(\xb,\xb^*+\nabla g(\xb))=\{\bar AL\mv L\in\Sp_f(\partial f)(\xb,\xba)\}$ follows from \cite[Lemma 3.1(iii)]{GfrOut22}. The formula for the adjoint SC derivative is a consequence of
   \begin{align*}\Sp_h^*(\partial h)(\xb,\xba+\nabla g(\xb))&=\{S(\bar A L)^\perp\mv L\in \Sp_f(\partial f)(\xb,\xba)\}=\{S\bar A^{-T} L^\perp\mv L\in \Sp_f(\partial f)(\xb,\xba)\}\\
   &=\{(S\bar A^{-T}S^{-1})(SL^\perp)\mv L\in \Sp_f(\partial f)(\xb,\xba)\}=\{\bar AL^*\mv L\in \Sp_f(\partial f)(\xb,\xba)\}.\end{align*}
\end{proof}

We will now show that for prox-regular functions the $f$-attentive objects defined in Definition \ref{DefAttDer} correspond to their ordinary counterparts applied to an $f$-attentive $\epsilon$-localization of the subdifferential.

\begin{definition}\label{DefEpsLocalization}Let $f:\R^n\to\oR$ be a function and $(\xb,\xba)\in \gph\partial  f$. Given $\epsilon>0$, the {\em $f$-attentive $\epsilon$-localization} of the subgradient mapping $\partial  f$ around $(\xb,\xba)$ is the multifunction $\Ta:\R^n\tto\R^n$ given by
\[\gph \Ta=\gph \partial_{f(\xb)+\epsilon}  f\cap (\B_\epsilon(\xb)\times \B_\epsilon(\xba)).\]
\end{definition}
We start our analysis with the following lemma inspired by \cite[Proposition 2.3]{PolRo96}. Recall that a  set $C\subset\R^n$ is said to be locally closed at a point $\xb$ (not necessarily in $C$) if
$C\cap \oB_\epsilon(\xb)$ is closed for some $\epsilon>0$.

\begin{lemma}\label{LemBasLemma}
  Assume that  the lsc function $f:\R^n\to\oR$ is prox-regular at $\xb\in\dom f $ for $\xba\in \partial f(\xb)$ with parameter values  $r\geq 0$ and $\epsilon>0$ and let $\Ta$ denote the $f$-attentive $\epsilon$-localization of $\partial f$ around $(\xb,\xba)$. Then the following statements hold.
  \begin{enumerate}
    \item[(i)] Let $(x_k,x_k^*)\in \gph \Ta$ be a sequence converging to some $(\hat x,\hat x^*)\in \B_\epsilon(\xb)\times \B_\epsilon(\xba)$. Then $f(\hat x)=\lim_k f(x_k)$ is finite and $(\hat x,\hat x^*)\in\gph \partial f$ and therefore $(x_k,x_k^*)\attconv{f}(\hat x,\hat x^*)$.
    \item[(ii)] Assume that $(x_k,x_k^*)\attconv{f}(\hat x,\hat x^*)\in \gph\Ta$. Then $(x_k,x_k^*)\in \gph\Ta$ for all $k$ sufficiently large.
    \item[(iii)] For every $(\hat x,\hat x^*)\in\gph \Ta$ and for every $\eta>0$ there exists some $\delta>0$ such that $f(x)<f(\hat x)+\eta$ for all $(x,x^*)\in\gph\Ta\cap(\B_\delta(\hat x)\times\B_\delta(\hat x^*))$.
    \item[(iv)] The set $\gph \Ta$ is locally closed  at each of its points.
  \end{enumerate}
\end{lemma}
\begin{proof}
ad (i): From the definition of prox-regularity we deduce that $f(\hat x)\geq f(x_k)+\skalp{x_k^*,\hat x-x_k}-\frac r2\norm{\hat x-x_k}^2$ and therefore $f(\hat x)\geq \limsup_k f(x_k)$. By lower semicontinuity of $f$ we also have $f(\hat x)\leq \liminf_k f(x_k)$ and $f(x_k)\to f(\hat x)\leq f(\xb)+\epsilon$ follows. The inclusion $\hat x^*\in\partial f(\hat x)$ now follows from the definition of the limiting subdifferential.

ad (ii): Since $(x_k,x_k^*)$ converges in $\gph\partial f$ to $(\hat x,\hat x^*)$ belonging to the open set $\B_\epsilon(\xb)\times \B_\epsilon(\xba)$, we have $(x_k,x_k^*)\in \gph \partial f\cap (\B_\epsilon(\xb)\times \B_\epsilon(\xba))$ for all $k$ sufficiently large. Further, since $f(x_k)\to f(\hat x)<f(\xb)+\epsilon$, we also have $f(x_k)<f(\xb)+\epsilon$ for all $k$ sufficiently large and  consequently $(x_k,x_k^*)\in\gph \Ta$ for these $k$.

ad (iii): By contraposition. Assume on the contrary that there is  $(\hat x,\hat x^*)\in\gph \Ta$ and  $\eta>0$ such that for every natural number $k$ there is some $(x_k,x_k^*)\in \gph \Ta\cap(\B_{1/k}(\hat x)\times\B_{1/k}(\hat x^*))$ with $f(x_k)\geq f(\hat x)+\eta$ and therefore $\liminf_k f(x_k)\geq f(\hat x)+\eta$. However, from (i) we conclude $f(\hat x)=\lim_k f(x_k)$, a contradiction.

ad (iv): Consider $(\hat x,\hat x^*)\in \gph \Ta$ and choose $\eta>0$ with $f(\hat x)+\eta<f(\xb)+\epsilon$ and $\delta>0$ small enough such that the compact neighborhood $C:=\oB_\delta(\hat x)\times \oB_\delta(\hat x^*)$ of $(\hat x,\hat x^*)$ is contained in $\B_\epsilon(\xb)\times \B_\epsilon(\xba)$ and $f(x)<f(\xb)+\eta$ for all $(x,x^*)\in\gph \Ta\cap C$. We now show that $\gph \Ta\cap C$ is closed.
  Consider a sequence $(x_k,x_k^*)\in \gph \Ta\cap C$ converging to some $(\tilde x,\tilde x^*)$. Clearly $(\tilde x,\tilde x^*)\in C$ and by (i) we have $f(\tilde x)=\lim_k f(x_k)\leq f(\hat x)+\eta<f(\xb)+\epsilon$ and $\tilde x^*\in \partial f(\tilde x)$ implying $(\tilde x,\tilde x^*)\in \gph \Ta\cap C$. Hence $\gph \Ta\cap C$ is closed and the lemma  is proved.
\end{proof}

\begin{proposition}\label{PropEpsLoc}Assume that the lsc function $f:\R^n\to\oR$ is prox-regular at $\xb$ for $\xba\in\partial f(\xb)$ with parameter values $r\geq0$ and $\epsilon>0$ and let $\Ta$ denote the $f$-attentive $\epsilon$-localization of $\partial f$ around $(\xb,\xba)$. Then for every $(\tilde x,\tilde x^*)\in\gph \Ta$ the following relations hold:
\begin{enumerate}
  \item[(i)]$T^f_{\gph \partial f}(\tilde x,\tilde x^*)=T_{\gph \Ta}(\tilde x,\tilde x^*)$, $\widehat N^f_{\gph \partial f}(\tilde x,\tilde x^*)=\widehat N_{\gph \Ta}(\tilde x,\tilde x^*)$, $N^f_{\gph \partial f}(\tilde x,\tilde x^*)=N_{\gph \Ta}(\tilde x,\tilde x^*)$.
  \item[(ii)]$D_f(\partial f)(\tilde x,\tilde x^*)=D\Ta(\tilde x,\tilde x^*)$, $\widehat D_f^*(\partial f)(\tilde x,\tilde x^*)=\widehat D^*\Ta(\tilde x,\tilde x^*)$, $D_f^*(\partial f)(\tilde x,\tilde x^*)=D^*\Ta(\tilde x,\tilde x^*)$.
  \item[(iii)]\[\emptyset\not=\Sp_f(\partial f)(\tilde x,\tilde x^*)=\Sp\Ta(\tilde x,\tilde x^*)=\Sp_f^*(\partial f)(\tilde x,\tilde x^*)=\Sp^*\Ta(\tilde x,\tilde x^*)\] and for every $L\in \Sp_f(\partial f)(\tilde x,\tilde x^*)$ there holds $L=L^*$ and there exist unique symmetric $n\times n$ matrices $P,W$ satisfying the properties (i)--(iii) of Proposition \ref{PropSCDProxRegSubDiff}.
\end{enumerate}
\end{proposition}
\begin{proof}
  ad (i): Consider $(u,u^*)\in T^f_{\gph \partial f}(\tilde x,\tilde x^*)$ together with sequences $t_k\downarrow 0$ und $(u_k,u_k^*)\to(u,u^*)$ with $(x_k,x_k^*):=(\tilde x+t_ku_k,\tilde x^*+t_ku_k^*)\attconv{f}(\tilde x,\tilde x^*)\in\gph \Ta$. By Lemma \ref{LemBasLemma} we have $(x_k,x_k^*)\in \gph \Ta$  for all $k$ sufficiently large and $(u,u^*)\in T_{\gph \Ta}(\tilde x,\tilde x^*)$ follows. This proves   $T^f_{\gph \partial f}(\tilde x,\tilde x^*)\subset T_{\gph \Ta}(\tilde x,\tilde x^*)$. In order to prove the reverse inclusion, consider $(u,u^*)\in T_{\gph \Ta}(\tilde x,\tilde x^*)$ together with sequences $t_k\downarrow 0$ and $(u_k,u_k^*)\to(u,u^*)$ with $(x_k,x_k^*):=(\tilde x+t_ku_k,\tilde x_k^*+t_ku_k^*)\longsetto{\gph\Ta}(\tilde x,\tilde x^*)$. From Lemma \ref{LemBasLemma} we conclude that $(x_k,x_k^*)\attconv{f}(\tilde x,\tilde x^*)$. Hence, $(u,u^*)\in T^f_{\gph \partial f}(\tilde x,\tilde x^*)$ and the equality $T^f_{\gph \partial f}(\tilde x,\tilde x^*)=T_{\gph \Ta}(\tilde x,\tilde x^*)$ is established. From this equality the second equality in (i) immediately follows from the definition.

  Now let us prove the third equality in (i). Consider $(u^*,u)\in N^f_{\gph \partial f}(\tilde x,\tilde x^*)$ together with sequences $(x_k,x_k^*)\attconv{f}(\tilde x,\tilde x^*)$ and $(u_k^*,u_k)\to(u^*,u)$ with $(u_k^*,u_k)\in \widehat N^f_{\gph \partial f}(x_k,x_k^*)$ for all $k$. Again by Lemma \ref{LemBasLemma} we have $(x_k,x_k^*)\in\gph \Ta$ for all $k$ sufficiently large implying, together with the just established equality for the regular normal cones, that $(u_k^*,u_k)\in \widehat N_{\gph \Ta}(x_k,x_k^*)$ for these $k$. Hence $(u^*,u)\in N_{\gph \Ta}(\tilde x,\tilde x^*)$ yielding $N^f_{\gph \partial f}(\tilde x,\tilde x^*)\subset N_{\gph \Ta}(\tilde x,\tilde x^*)$. Conversely, consider $(u^*,u)\in N_{\gph \Ta}(\tilde x,\tilde x^*)$ together with sequences $(x_k,x_k^*)\longsetto{\gph \Ta}(\tilde x,\tilde x^*)$ and $(u_k^*,u_k)\to(u^*,u)$ with $(u_k^*,u_k)\in \widehat N_{\gph \Ta}(x_k,x_k^*)$ for all $k$. Since $(x_k,x_k^*)\attconv{f} (\tilde x,\tilde x^*)$ by Lemma \ref{LemBasLemma}, we obtain $(u^*,u)\in N^f_{\gph \partial f}(\tilde x,\tilde x^*)$ showing the third equality in (i).

  ad (ii): This is an immediate consequence of (i).

  ad (iii): Consider $L\in \Sp_f(\partial f)(\tilde x,\tilde x^*)$ together with a sequence $(x_k,x_k^*)\in\OO_{\partial f}^f$ satisfying $(x_k,x_k^*)\attconv{f}(\tilde x,\tilde x^*)$ and
  $\lim_k d_{\Z}(L,T_{\gph \partial f}^f(x_k,x_k^*))=0$. By Lemma \ref{LemBasLemma} we have $(x_k,x_k^*)\in\gph \Ta$ for all $k$ sufficiently large and, since $T_{\gph \partial f}^f(x_k,x_k^*)=T_{\gph\Ta}(x_k,x_k^*)$ is a subspace of dimension $n$, $L \in \Sp\Ta(\tilde x,\tilde x^*)$ follows. Conversely, if $L \in \Sp\Ta(\tilde x,\tilde x^*)$ consider a sequence $(x_k,x_k^*)\longsetto{\OO_{\Ta}}(\tilde x,\tilde x^*)$ with $\lim_k d_{\Z}(L,T_{\gph\Ta}(x_k,x_k^*))=0$. Since we have $(x_k,x_k^*)\attconv{f}(\tilde x,\tilde x^*)$ by Lemma \ref{LemBasLemma} and $T_{\gph \partial f}^f(x_k,x_k^*)=T_{\gph\Ta}(x_k,x_k^*)$ is a subspace of dimension $n$, we conclude that $(x_k,x_k^*)\longsetto{\OO_{\partial f}^f}(\tilde x,\tilde x^*)$ and $L\in \Sp_f(\partial f)(\tilde x,\tilde x^*)$ follows. Thus we have shown $\Sp_f(\partial f)(\tilde x,\tilde x^*)=\Sp\Ta(\tilde x,\tilde x^*)$ and the equation
  $\Sp_f^*(\partial f)(\tilde x,\tilde x^*)=\Sp^*\Ta(\tilde x,\tilde x^*)$ follows from the definition. Next choose $\epsilon'>0$ small enough such that $f(\tilde x)+\epsilon'\leq f(\xb)+\epsilon$ and $\B_{\epsilon'}(\tilde x)\times \B_{\epsilon'}(\tilde x^*)\subset \B_\epsilon(\xb)\times\B_\epsilon(\xba)$. Then  $f$ is prox-regular at $\tilde x$ for $\tilde x^*$ with parameter values $r$ and any $\tilde\epsilon\in(0,\epsilon']$ and, by Remark \ref{RemVarConv}(i), there is an lsc function $\widehat f$ together with an open convex neighborhood $U\times V$ of $(\tilde x,\tilde x^*)$ and some real $\rho>f(\tilde x)$ such that $\widehat f\leq f$ on $U$,  $\widehat f+\frac r2\norm{\cdot}^2$ is convex on $U$ and
  \[\gph\partial_\rho f\cap(U\times V)= \gph\partial\widehat f\cap(U\times V)\quad \mbox{and $f(x)=\widehat f(x)$ at the common elements $(x,x^*)$.}\]
  Next choose $\tilde\epsilon\in(0,\epsilon']$ with $\tilde\rho:=f(\tilde x)+\tilde\epsilon\leq\rho$. By applying our just proven result with the $f$-attentive $\tilde\epsilon$-localization $\tilde \T_{\tilde \epsilon}^f$ around $(\tilde x,\tilde x^*)$, we  obtain that $\Sp_f(\partial f)(\tilde x,\tilde x^*)= \Sp \tilde\T_{\tilde\epsilon}^f(\tilde x,\tilde x^*)$ and $\Sp^*_f(\partial f)(\tilde x,\tilde x^*)= \Sp^* \tilde\T_{\tilde\epsilon}^f(\tilde x,\tilde x^*)$. Further,  $\widehat f$ is prox-regular and subdifferentially continuous at $\tilde x$ for $\tilde x^*$ and we may therefore find a convex neighborhood $\tilde U\times\tilde V\subset (U\times V)\cap (\B_{\tilde\epsilon}(\tilde x)\times \B_{\tilde\epsilon}(\tilde x^*))$ such that $f(x)=\widehat f(x)<\tilde\rho$ $\forall (x,x^*)\in\gph\partial\widehat f\cap(\tilde U\times \tilde V)$. Thus
  \[\gph \partial \widehat f\cap(\tilde U\times\tilde V)= \gph \partial_{\tilde\rho} f\cap(\tilde U\times\tilde V)=
  \gph \tilde\T_{\tilde \epsilon}^f\cap (\tilde U\times\tilde V)\]
  and, together with Proposition \ref{PropSCDProxRegSubDiff}, we conclude that
  \begin{align*}\Sp_f (\partial f)(\tilde x, \tilde x^*)&=\Sp \tilde\T_{\tilde \epsilon}^f(\tilde x,\tilde x^*)=\Sp(\partial\widehat f)(\tilde x,\tilde x^*)\\
  &=\Sp^*(\partial\widehat f)(\tilde x,\tilde x^*)= \Sp^* \tilde\T_{\tilde \epsilon}^f(\tilde x,\tilde x^*)= \Sp_f^* (\partial f)(\tilde x, \tilde x^*)\not=\emptyset.\end{align*}
 Further, for every subspace $L\in\Sp_f (\partial f)(\tilde x, \tilde x^*)=\Sp(\partial\widehat f)(\tilde x,\tilde x^*)$ the property $L=L^*$ and the claimed basis representation for $L$ follows also from Proposition \ref{PropSCDProxRegSubDiff}.
  \end{proof}
\begin{remark}Although we could equivalently work with the classical generalized derivatives for the $f$-attentive $\epsilon$-localizations of $\partial f$ instead of the $f$-attentive ones , the use of the latter has the advantage that it is independent of $\epsilon$.
\end{remark}

In what follows, we denote by $\M^f_{P,W}\partial f(x,x^*)$ the collection of all pairs of symmetric $n\times n$ matrices $(P,W)$ satisfying the properties (ii) and (iii) of Proposition \ref{PropSCDProxRegSubDiff} and $\rge(P,W)\in \Sp_f(\partial f)(x,x^*)$.

Note that for a twice continuously differentiable function $g:\R^n\to\R$ we have by virtue of Proposition \ref{PropCalc} that
\[\Sp_{f+g}(\partial(f+g))(\xb,\xba+\nabla g(\xb))=\{\rge(P,\nabla\partial^2g(\xb)P+W)\mv (P,W)\in\M_{P,W}^f\partial f(\xb,\xba)\}.\]
For every $(P,W)\in \M_{P,W}^f\partial f(\xb,\xba)$ and every $p\in\R^n$ there holds
\[\big(Pp,(\nabla\partial^2g(\xb)P+W)p\big)=\big(P\bar p,(P\nabla^2g(\xb)P+PWP+(I-P))\bar p\big)\]
 with $\bar p=p+(I-P)\nabla^2g(\xb)Pp$. Thus $\rge(P,\nabla\partial^2g(\xb)P+W)=\rge(P,P\nabla\partial^2g(\xb)P+PWP +(I-P))$ and we obtain that
\begin{equation}\label{EqSumRule}\M_{P,W}^{f+g}\partial(f+g)(\xb,\xba+\nabla g(\xb))=\{(P,P\nabla^2g(\xb)P+W)\mv (P,W)\in\M_{P,W}^f\partial f(\xb,\xba)\}.\end{equation}

At the end of this section we state the following result which confirms that under prox-regularity and subdifferential continuity the $f$-attentive generalized derivatives coincide locally with their ordinary counterparts.
\begin{lemma}
  Assume that the lsc function $f:\R^n\to\oR$ is prox-regular and subdifferentially continuous at $\xb\in\dom f$ for the subgradient $\xba\in\partial f(\xb)$. Then there is a neighborhood $U\times V$ of $(\xb,\xba)$  such that for every $(x,x^*)\in \gph\partial f\cap (U\times V)$ the function $f$ is prox-regular and subdifferentially continuous at $x$ for $x^*$ and
  \begin{gather*}T_{\gph \partial f}^f(x,x^*)= T_{\gph \partial f}(x,x^*),\ \widehat N_{\gph \partial f}^f(x,x^*)= \widehat N_{\gph \partial f}(x,x^*),\ N_{\gph \partial f}^f(x,x^*)= N_{\gph \partial f}(x,x^*),\\
 D_f(\partial f)(x,x^*)=D(\partial f)(x,x^*),\ \widehat D^*_f(\partial f)(x,x^*)=\widehat D^*(\partial f)(x,x^*),\ D^*_f(\partial f)(x,x^*)=\widehat D^*(\partial f)(x,x^*),\\
   \Sp_f(\partial f)(x,x^*)=\Sp(\partial f)(x,x^*),\ \Sp^*_f(\partial f)(x,x^*)=\Sp^*(\partial f)(x,x^*).
 \end{gather*}
\end{lemma}
\begin{proof}
  Let $r\geq 0$ and $\epsilon>0$ denote the parameter values for prox-regularity of $f$ at $\xb$ for $\xba$. By subdifferential continuity we can find some $\epsilon'\in(0,\epsilon]$ such that for every $(x,x^*)\in \gph\partial f\cap(\B_{\epsilon'}(\xb)\times \B_{\epsilon'}(\xba))$ one has $f(x)<f(\xb)+\epsilon$. Consider $(x,x^*)\in \gph\partial f\cap(\B_{\epsilon'}(\xb)\times \B_{\epsilon'}(\xba))$. It is easy to see that $f$ is prox-regular at $x$ for $x^*$ with parameters $r$ and any $\tilde\epsilon>0$ chosen small enough so that $\B_{\tilde\epsilon}(x)\times \B_{\tilde\epsilon}(x^*)\subset \B_{\epsilon}(\xb)\times \B_{\epsilon}(\xba)$ and $f(x)+\tilde\epsilon<f(\xb)+\epsilon$. In order to verify subdifferential continuity, consider a sequence $(x_k,x_k^*)\longsetto{\gph\partial f}(x,x^*)$. For $k$ sufficiently large we have $(x_k,x_k^*)\in \gph\partial f\cap(\B_{\epsilon'}(\xb)\times \B_{\epsilon'}(\xba))\subset \gph\Ta$ and $f(x_k)\to f(x)$ follows from Lemma \ref{LemBasLemma}(i). This verifies subdifferential continuity of $f$ at $x$ for $x^*$ and the remaining assertion is a consequence of the discussion following Definition \ref{DefAttDer}.
\end{proof}

\section{On neighborhood-based second-order characterizations of  variational $s$-convexity}
We start our analysis with two preparatory lemmata.
\begin{lemma}\label{LemLowerFuncVarConv}
  Assume that we are given two lsc functions $f,\psi:\R^n\to\oR$ together with a point $\xb\in\dom f$ and a subgradient $\xba\in\partial f(\xb)$. Further assume that we are given a neighborhood $U\times V$ of $(\xb,\xba)$ together with some $\rho>f(\xb)$ such that $\psi\leq f$ on $U$ and
  \[\gph\partial_\rho f\cap(U\times V)=\gph \partial\psi\cap(U\times V)\quad \mbox{and $f(x)=\psi(x)$ at the common elements $(x,x^*)$}.\]
  Then $f$ is variationally $s$-convex at $\xb$ for $\xba$, provided $\psi$ has this property.
\end{lemma}
\begin{proof}
  Assume that $\psi$ is variationally $s$-convex at $\xb$ for $\xba$ and consider an open convex neighborhood $\widehat U\times\widehat V$ of $(\xb,\xba)$  together with a  function $\widehat \psi\leq \psi$ on $\widehat U$   and some $\widehat \rho>0$ such that $\widehat\psi-\frac s2\norm{\cdot}^2$ is convex on $\widehat U$ and
  \[\gph\partial_{\widehat \rho} \psi\cap(\widehat U\times \widehat V)=\gph \partial\widehat \psi\cap(\widehat U\times \widehat V)\quad \mbox{and $\psi(x)=\widehat \psi(x)$ at the common elements $(x,x^*)$}.\]
  Taking $\tilde\rho:=\min\{\widehat\rho,\rho\}$, by Remark \ref{RemVarConv}(ii) we can find an open convex neighborhood $\tilde U\times \tilde V\subset \widehat U\times\widehat V$ of $(\xb,\xba)$ such that
  \[\gph\partial_{\tilde\rho} \psi\cap(\tilde U\times \tilde V)=\gph \partial\widehat\psi\cap(\tilde U\times  \tilde V)\quad \mbox{and $\psi(x)=\widehat \psi(x)$ at the common elements $(x,x^*)$}\]
  and, by  possibly shrinking $\tilde U\times \tilde V$, we may assume that $\tilde U\times \tilde V\subset U\times V$. Then for any $(x,x^*)\in\gph \partial \widehat \psi\cap(\tilde U\times \tilde V)$ we have $(x,x^*)\in \gph \partial_{\tilde\rho}\psi\cap (U\times V)$ implying $\widehat \psi(x)=\psi(x)=f(x)<\tilde\rho$ and $\gph \partial \widehat \psi\cap(\tilde U\times \tilde V)=\gph \partial \widehat \psi\cap(\tilde U\times \tilde V)\cap(L_f(\tilde\rho)\times\R^n)$. Since $\psi\leq f$ on $\tilde U\subset U$, we have $L_\psi(\tilde\rho)\cap \tilde U\supset L_f(\tilde\rho)\cap \tilde U$ and, together with $L_f(\rho)\supset L_f(\tilde\rho)$, we obtain that
  \begin{align*}\gph \partial \widehat \psi\cap(\tilde U\times \tilde V)&= \gph \partial\psi\cap (L_\psi(\tilde\rho)\times\R^n)\cap(\tilde U\times \tilde V)\cap (L_f(\tilde\rho)\times\R^n)\\
  &= \gph\partial f\cap (L_f(\rho)\times\R^n)\cap (\tilde U\times \tilde V)\cap (L_f(\tilde\rho)\times\R^n)=\gph\partial_{\tilde\rho} f\cap (\tilde U\times \tilde V)\end{align*}
  showing variational $s$-convexity of $f$ at $\xb$ for $\xba$.
\end{proof}
\begin{lemma}\label{LemQuadrShift}Let $s\in\R$ and assume that the lsc function $f:\R^n\to\oR$ is variationally $s$-convex at $\xb\in\dom f$ for $\xba\in\partial f(\xb)$. Then for every $t\in\R$ and every $y^*\in\R^n$ the function
\[\psi:=f+\skalp{y^*,\cdot-\xb}+\frac t2\norm{\cdot-\xb}^2\]
is variationally $(s+t)$-convex at $\xb$ for $\xba+y^*$. Moreover, if $\widehat f$, $U$, $V$ and $\rho$ are as in Definition \ref{DefVarConv}, then variational $(s+t)$-convexity of $\psi$ can be certified by the function $\widehat \psi:=\widehat f+\skalp{y^*,\cdot-\xb}+\frac t2\norm{\cdot-\xb}^2$ together with an suitable open convex neighborhood $\widehat U\times\widehat V$ of $(\xb,\xba+y^*)$ and some $\widehat\rho>\psi(\xb)=f(\xb)$.
\end{lemma}
\begin{proof}
  Let $y^*\in\R^n$ and $t\in\R$ be arbitrarily fixed and note that $\partial\psi(x)=\partial f(x)+y^*+t(x-\xb)$. Next we choose some sufficiently small neighborhood $\widehat U\times\widehat V$ of $(\xb,\xba+y^*)$ such that $\widehat U\subset U$, $\widehat V-y^*-t(U-\xb)\subset V$ and $\skalp{y^*,x-\xb}+\frac t2\norm{x-\xb}^2 >-\epsilon$, $x\in\widehat U$, where $\epsilon:=(\rho-f(\xb))/2$. Since $\widehat f\leq f$ on $U$, it follows that $\widehat\psi\leq\psi$ on $U\supset\widehat U$. Further, since $\widehat f-\frac s2\norm{\cdot}^2$ is convex on $U$, $\widehat \psi-\frac{s+t}2\norm{\cdot}^2$ is also convex there and we can conclude that $\widehat\psi$ is  subdifferentially continuous at $\xb$ for $\xba+y^*$. Hence, by possibly shrinking $\widehat U\times\widehat V$, we may assume that $\widehat\psi(x)<\widehat\rho:=f(\xb)+\epsilon$ whenever $(x,x^*)\in \widehat U\times\widehat V$.
  Then for every $(x,x^*)\in \gph \partial\widehat\psi\cap(\widehat U\times\widehat V)$ we have $(x,x^*-y^*-t(x-\xb))\in\gph \partial\widehat f\cap(U\times V)$ and therefore
  $(x,x^*-y^*-t(x-\xb))\in \gph \partial f$ and $f(x)=\widehat f(x)$. Hence $(x,x^*)\in\gph \partial\psi(x)\cap(\widehat U\times\widehat V)$ and $\psi(x)=\widehat \psi(x)<\widehat\rho$.
  On the other hand, if $(x,x^*)\in \gph \partial_{\widehat \rho}\psi(x)\cap(\widehat U\times\widehat V)$ we obtain that
  \[f(x)=\psi(x)-\skalp{y^*,x-\xb}-\frac t2\norm{x-\xb}^2<\psi(x)+\epsilon<\hat\rho+\epsilon=\rho\]
   and $(x,x^*-y^*-t(x-x^*))\in\gph\partial f\cap(U\times V)$. Hence, $(x,x^*-y^*-t(x-\xb))\in\gph\partial\widehat f\cap (U\times V)$ implying $(x,x^*)\in\gph \partial\widehat\psi\cap(\widehat U\times\widehat V)$. Thus we have shown
   \[\gph\partial_{\widehat \rho}\psi\cap(\widehat U\times\widehat V)= \gph\partial\widehat\psi\cap(\widehat U\times\widehat V)\quad\mbox{and $\psi(x)=\widehat \psi(x)$ at the common elements $(x,x^*)$}\]
   and the proof is complete.
\end{proof}
Proposition \ref{PropEpsLoc} is the basis of the following neighborhood-based characterizations of variational convexity.
\begin{theorem}\label{ThCharVarConv}
  Let $f:\R^n\to\oR$ be an lsc function and let $\xb\in\dom f$, $\xba\in\partial f(\xb)$ and some real $s$ be given. Then the following statements are equivalent:
  \begin{enumerate}
    \item[(i)] $f$ is variationally $s$-convex at $\xb$ for $\xba$.
    \item[(ii)] $f$ is prox-regular at $\xb$ for $\xba$ and there exists a neighborhood $U\times V$ of $(\xb,\xba)$  along with some $\rho>f(\xb)$ such that
        \begin{equation}\label{EqCharVarConv1}
          \skalp{z^*,z}\geq s\norm{z}^2\quad\mbox{ whenever }\quad z^*\in \widehat D_f^*(\partial f)(x,x^*)(z),\ (x,x^*)\in\gph \partial_\rho f\cap(U\times V).
        \end{equation}
    \item[(iii)] $f$ is prox-regular at $\xb$ for $\xba$ and there exists a neighborhood $U\times V$ of $(\xb,\xba)$ along with some $\rho>f(\xb)$ such that
        \begin{equation}\label{EqCharVarConv2}
          \skalp{z^*,z}\geq s\norm{z}^2\quad\mbox{ whenever }\quad z^*\in D_f^*(\partial f)(x,x^*)(z),\ (x,x^*)\in\gph \partial_\rho f\cap(U\times V).
        \end{equation}
    \item[(iv)] $f$ is prox-regular at $\xb$ for $\xba$ and there exists a neighborhood $U\times V$ of $(\xb,\xba)$ along with some $\rho>f(\xb)$ such that
        \begin{equation}\label{EqCharVarConv3}
          \skalp{z^*,z}\geq s\norm{z}^2\quad\mbox{ whenever }\quad (z,z^*)\in T^f_{\gph\partial f}(x,x^*),\ (x,x^*)\in\OO^f_{\partial f}\cap(U\times V),\ f(x)<\rho.
        \end{equation}
    \item[(v)] $f$ is prox-regular at $\xb$ for $\xba$ and there exists a neighborhood $U\times V$ of $(\xb,\xba)$ along with some $\rho>f(\xb)$ such that
        \begin{equation}\label{EqCharVarConv4}
          \skalp{z^*,z}\geq s\norm{z}^2\quad\mbox{ whenever }\quad (z,z^*)\in L\in \Sp_f(\partial f)(x,x^*),\ (x,x^*)\in\gph \partial_\rho f\cap(U\times V).
          \end{equation}
\if{     \item[(vi)] $f$ is prox-regular at $\xb$ for $\xba$ and there exists a neighborhood $U\times V$ of $(\xb,\yb)$ along with some $\rho>f(\xb)$ such that
    \begin{align*} \rge(P,W)= T^f_{\gph\partial f}(x,x^*)\ &\Rightarrow\ \mbox{$PWP-sP$ is pos. semidefinite}\end{align*}
     whenever $(x,x^*)\in \OO^f_{\partial f} f\cap(U\cap V)$ and $f(x)<\rho$.
     }\fi
    \item[(vi)] $f$ is prox-regular at $\xb$ for $\xba$ and there exists a neighborhood $U\times V$ of $(\xb,\xba)$ along with some $\rho>f(\xb)$ such that
    \begin{align*}  (P,W)\in\M^f_{P,W}\partial f(x,x^*)\ \Rightarrow\ \mbox{$PWP-sP$ is pos. semidefinite}\end{align*}
    whenever $(x,x^*)\in \gph \partial_\rho f\cap(U\cap V)$.
  \end{enumerate}
\end{theorem}
\begin{proof}By Lemma \ref{LemQuadrShift} and Proposition \ref{PropCalc} together with \eqref{EqSumRule}, it suffices to prove the theorem for $s=0$.
Assume first that $f$ is variationally convex at $\xb$ for $\xba$ and consider the lsc function $\widehat f$ together with some open convex $\widehat U\times \widehat V$ of $(\xb,\xba)$ and some $\widehat\rho>f(\xb)$ such that
$\widehat f$ is convex on $\widehat U$ and
\begin{equation}\label{EqVarConvHat}
  \widehat f\leq f\mbox{ on }\widehat U,\ \gph\partial_{\widehat\rho} f\cap(\widehat U\times\widehat V)=\gph \partial \widehat f\cap(\widehat U\times\widehat V)\mbox{ and $f(x)=\widehat f(x)$ at the common elemnts $(x,x^*)$.}
\end{equation}
As pointed out in Remark \ref{RemVarConv}(i), variational convexity of $f$ implies prox-regularity with parameter $r=0$ and some $\epsilon>0$. By possibly shrinking $\epsilon$ we may assume that $\B_\epsilon(\xb)\times\B_\epsilon(\xba)\subset \widehat U\times\widehat V$ and $\rho:=f(\xb)+\epsilon\leq\widehat \rho$.
The function $\widehat f$ is convex on $\widehat U$ and consequently variationally convex, prox-regular and subdifferentially continuous at $\xb$ for $\xba$. Thus we may use the characterization \cite[Theorem 5.2]{KhMoPh23a} of variational convexity to obtain the existence of a neighborhood $U\times V$ of $(\xb,\xba)$ such that
\begin{equation}\label{EqKhMoPhCrit}\skalp{z^*,z}\geq 0\quad\mbox{whenever}\quad z^*\in\Phi(x,x^*)(z),\ (x,x^*)\in\gph\partial\widehat f\cap(U\times V),\ z\in\R^n,\end{equation}
where $\Phi(x,x^*)$ is either $\widehat D^*(\partial \widehat f)(x,x^*)$ or $D^*(\partial \widehat f)(x,x^*)$.
By possibly shrinking the neighborhood $U\times V$, we may assume that $U\times V\subset\B_\epsilon(\xb)\times \B_\epsilon(\xba)$ and $\widehat f(x)<\rho$ whenever $(x,x^*)\in \gph\partial\widehat f\cap (U\times V)$. Thus for every $(x,x^*)\in \gph\partial\widehat f\cap(U\times V)$ we have $(x,x^*)\in\gph \partial f$ and  $f(x)=\widehat f(x)<f(\xb)+\epsilon$ implying $\gph\partial\widehat f\cap(U\times V)\subset\gph\Ta\cap(U\times V)$. On the other hand, for every $(x,x^*)\in \gph\Ta\cap(U\times V)\subset \gph\partial_{\widehat \rho} f\cap(\widehat U\times\widehat V)$ we have $(x,x^*)\in \gph\partial\widehat f\cap(U\times V)$ and $\gph\partial\widehat f\cap(U\times V)=\gph\Ta\cap(U\times V)$ follows. Hence $\widehat D^*(\partial \widehat f)(x,x^*)=\widehat D^*\Ta(x,x^*)$ and  $D^*(\partial \widehat f)(x,x^*)=D^* \Ta(x,x^*)$, $(x,x^*)\in\gph \Ta\cap(U\times V)$ and the statements (ii) and (iii) follow from  \eqref{EqKhMoPhCrit} and Proposition \ref{PropEpsLoc}.

Next assume that (ii) or (iii) hold with $s=0$. Since $f$ is assumed to be prox-regular at $\xb$ for $\xba$ with some parameter values $r\geq 0$ and $\epsilon>0$, by taking into account Remark \ref{RemVarConv}(i), $f$ is also variationally $(-r)$-convex there and hence there exists an lsc function $\widehat f$  together with some open convex $\widehat U\times \widehat V$ of $(\xb,\xba)$ and some $\widehat\rho>f(\xb)$ such that
$\widehat f+\frac r2\norm{\cdot}^2$ is convex on $\widehat U$ and \eqref{EqVarConvHat} holds. By possibly shrinking the neighborhood $U\times V$ in (ii) or (iii) and lowering $\rho$ and $\epsilon$, by utilizing subdifferential continuity of $\widehat f$ at $\xb$ for $\xba$  we can assume  that $U \times V\subset\B_\epsilon(\xb)\times \B_\epsilon(\xba)\subset \widehat U\times\widehat V$ and $\widehat f(x)<\rho\leq f(\xb)+\epsilon\leq\widehat\rho$ whenever $(x,x^*)\in\gph\partial\widehat f\cap(U\times V)$. Using the same arguments as above, we conclude that $\gph\partial\widehat f\cap(U\times V)=\gph\Ta\cap(U\times V)=\gph\partial_\rho f\cap(U\times V)$ and $\widehat D^*(\partial \widehat f)(x,x^*)=\widehat D^*\Ta(x,x^*)=\widehat D_f^*(\partial f)(x,x^*)$ and  $D^*(\partial \widehat f)(x,x^*)=D^* \Ta(x,x^*)=D_f^*(\partial f)(x,x^*)$ for all $(x,x^*)\in \gph\partial_\rho f\cap(U\times V)$ follows. Hence \eqref{EqKhMoPhCrit} is valid and, since $\widehat f$ is prox-regular and subdifferentially continuous at $\xb$ for $\xba$, by \cite[Theorem 5.2]{KhMoPh23a} we obtain that $\widehat f$ is variationally convex at $\xb$ for $\xba$. Now (i) follows from Lemma \ref{LemLowerFuncVarConv}.

We proceed by showing the equivalence between (i) and (iv). In order to do this we will use the equivalence (i)$\Leftrightarrow$(iii) in Theorem \ref{ThFundEquiv}.
\if{Note that $\xba\in\widehat f(\xb)$ holds because $f$ is prox-regular at $\xb$ for $\xba$, either as  a consequence of variational convexity or by assumption.}\fi

If (i) holds then  $\partial f$ is $f$-locally monotone around $(\xb,\xba)$ and hence  $\Ta$ is  monotone around $(\xb,\xba)$ for all sufficiently small $\epsilon>0$  and  by \cite[Corollary 5.8]{PolRo96} this is equivalent with the existence of some neighborhood $U\times V$ of $(\xb,\xba)$ such that for all $(x,x^*)\in \OO_{\Ta}\cap(U\times V)$ the graphical derivative  $D\Ta(x,x^*)$ is a monotone mapping. In order to apply \cite[Corollary 5.8]{PolRo96} note that at points $(x,x^*)\in \OO_{\Ta}\cap(U\times V)$ the mapping $\Ta$ is in fact proto-differentiable because of  \cite[Proposition 13.46]{RoWe98} and \cite[Remark 3.18]{GfrOut22}.
From Proposition \ref{PropEpsLoc} we conclude that $\gph D\Ta(x,x^*)=T_{\gph \Ta}(x,x^*)=T_{\gph\partial f}^f(x,x^*)\in\Z_{nn}$  for every $(x,x^*)\in \OO_{\Ta}\cap(U\times V)$ and thus $(x,x^*)\in\OO^f_{\partial f}\cap((U\cap L_f(f(\xb)+\epsilon))\times V)$. Further, monotonicity of $D\Ta(x,x^*)$ implies $\skalp{z^*,z}\geq0$, $(z,z^*)\in\gph D\Ta(x,x^*)=T_{\gph\partial f}^f(x,x^*)$. This proves the implication (i)$\Rightarrow$(iv) with $\rho=f(\xb)+\epsilon$.

 Conversely, if (iv) holds then we may choose the parameter $\epsilon$ for prox-regularity so small  that $\B_\epsilon(\xb)\times\B_\epsilon(\xba)\subset U\times V$ and $f(\xb)+\epsilon<\rho$. Hence, $(x,x^*)\in \OO_{\Ta}\cap(U\times V)$ implies $(x,x^*)\in\OO^f_{\partial f}\cap((U\cap L_f(\rho))\times V)$ and $\gph D\Ta(x,x^*)=T_{\gph\partial f}^f(x,x^*)\in\Z_{nn}$ by virtue of Proposition \ref{PropEpsLoc}. Since $\gph D\Ta(x,x^*)$ is a subspace, it follows from \eqref{EqCharVarConv3} that $D\Ta(x,x^*)$ is monotone for these $(x,x^*)$ and consequently $\Ta$ is locally monotone at $(\xb,\xba)$ by  \cite[Corollary 5.8]{PolRo96} which implies that $\partial f$ has an $f$-attentive monotone localization around $(\xb,\xba)$ and consequently $f$ is variationally convex at $\xb$ for $\xba$.

Next we prove that (iv) is equivalent with (v). Obviously, (v) implies (iv). Now assume that (iv) holds with some open convex neighborhood $U\times V$ of $(\xb,\xba)$ and some $\rho>f(\xb)$ and consider at a point $(x,x^*)\in \gph\partial_\rho f\cap(U\times V)$ and a subspace $L\in \Sp_f\partial f(x,x^*)$ together with a sequence $(x_k,x_k^*)\subset \OO^f_{\partial f}$ satisfying $(x_k,x_k^*)\attconv{f}(x,x^*)$ and  $d_{\Z}(L,L_k)\to0$, where $L_k:=T^f_{\gph \partial f}(x_k,x_k^*)$. Denoting by $P$ and $P_k$ the orthogonal projections onto $L$ and $L_k$, respectively, we obtain that for every $(z,z^*)\in L$ the sequence $(z_k,z_k^*):=P_k(z,z^*)\in L_k$ converges to $P(z,z^*)=(z,z^*)$. For all $k$ sufficiently large we have $(x_k,x_k^*)\in\OO^f_{\partial f}\cap((U\cap L_f(\rho))\times V)$ and therefore $\skalp{z_k^*,z_k}\geq 0$. This implies $\skalp{z^*,z}=\lim_{k\to\infty}\skalp{z_k^*,z_k}\geq 0$ and the implication (iv)$\Rightarrow$(v) is established.

In order to show the equivalence between (v) and (vi), note that for $(x,x^*)\in\gph \partial_\rho f\cap(U\times V)$ and $L=\rge(P,W)\in\Sp_f(\partial f(x,x^*)$ with $(P,W)\in\M_{P,W}^f\partial f(x,x^*)$ the condition $\skalp{z^*,z}\geq 0$, $(z,z^*)\in L$ is equivalent with
 \[0\leq\skalp{Wp,Pp}=\skalp{p,WPp}=\skalp{p,PWPp},\ p\in\R^n\]
 and the equivalence between (v) and (vi) is established for $s=0$.
\end{proof}
\begin{remark}
  Very recently, Khanh, Khoa, Mordukhovich and Phat \cite[Theorem 5.5]{KhKhMoPh23c} gave a neigh\-bor\-hood-based characterization of variational convexity in terms of  regular coderivatives as well as limiting coderivatives of the $f$-attentive $\epsilon$-localization of $\partial f$ around $(\xb,\xba)$. By taking into account Proposition \ref{PropEpsLoc} and Lemma \ref{LemQuadrShift}, the equivalences (i)$\Leftrightarrow$(ii)$\Leftrightarrow$(iii) in Theorem \ref{ThCharVarConv} can also be derived from this result.
\end{remark}
\begin{remark}
  Consider $(P,W)\in \M^f_{P,W}\partial f(x,x^*)$. Taking into account \eqref{EqPW3}, it is easy to see that the matrix $PWP$ is positive semidefinite if and only if $W$ has this property.
\end{remark}

\section{Point-based second-order characterizations of variational strong convexity and tilt-stability}

We proceed by establishing equivalence between variational strong convexity and tilt-stability, present point-based second-order characterizations of these properties  and compute the exact bound for variational convexity and tilt-stability, respectively.

\begin{theorem}\label{ThCharStrVarConv}
  Let $f:\R^n\to\oR$ be an lsc function and let $\xb\in\dom f$ and $\xba\in\partial f(\xb)$ together with the real number $\sigma>0$ be given. Then the following statements are equivalent:
  \begin{enumerate}
    \item[(i)]  For every $\sigma'\in(0,\sigma)$, $f$ is  variationally strongly convex  at $\xb$ for $\xba$ with modulus $\sigma'$.
    \item[(ii)] $f$ is prox-regular at $\xb$ for $\xba$ and  for every $\sigma'\in(0,\sigma)$, $\xb$ is a tilt-stable local minimizer for $f-\skalp{\xba,\cdot}$ with modulus ${\sigma'}^{-1}$.
    \item[(iii)] $f$ is prox-regular at $\xb$ for $\xba$ and
        \begin{equation}\label{EqCharStrVarConv1}
          \skalp{z^*,z}\geq\sigma\norm{z}^2\quad\mbox{ whenever }\quad z^*\in D_f^*(\partial f)(\xb,\xba)(z),\ z\in\R^n.
        \end{equation}
        \item[(iv)] $f$ is prox-regular at $\xb$ for $\xba$ and
        \begin{equation}\label{EqCharStrVarConv2} \skalp{p,PWPp}\geq \sigma\norm{Pp}^2\quad\mbox{ whenever }\quad (P,W)\in\M^f_{P,W}\partial f(\xb,\xba),\ p\in\R^n.
        \end{equation}
    \item[(v)] $f$ is prox-regular at $\xb$ for $\xba$ and for every pair $(P,W)\in\M^f_{P,W}\partial f(\xb,\xba)$ the matrix $W-\sigma P$ is positive semidefinite.
  \end{enumerate}
  Finally, if $f$ is variationally strongly convex at $\xb$ for $\xba$, we have the following formulas for the exact bound of  variational convexity and tilt-stability:
  \begin{align}
    \label{EqTilt1}\varco(f;\xb|\xba)^{-1}&=\tilt(f-\skalp{\xba,\cdot};\xb)=\sup\left\{\frac{\norm{z}^2}{\skalp{z^*,z}}\mv z^*\in D^*_f(\partial f)(\xb,\xba)(z),\ z\in\R^n\right\}\\
    \label{EqTilt2}&=\sup\left\{\frac{\norm{Pp}^2}{\skalp{p,PWPp}}\mv (P,W)\in M_{P,W}^f\partial f(\xb,\xba), p\in\R^n\right\}\\
    \label{EqTilt3}&=\sup\{\norm{PW^{-1}}\mv (P,W)\in M_{P,W}^f\partial f(\xb,\xba)\}
  \end{align}
  with the convention $\frac 00:=0$.
\end{theorem}
\begin{proof}By taking into account Proposition \ref{PropCalc} and Lemma \ref{LemLowerFuncVarConv}, in order to unburden notation, we may assume that $\xba=0$.

ad ''(i)$\Rightarrow$(ii)'':  The statement follows from \cite[Theorem 3]{Ro19} together with the fact that variational convexity implies prox-regularity, cf. Remark \ref{RemVarConv}(i).

ad ''(ii)$\Rightarrow$(iii)'':  Assume that $f$ is prox-regular at $\xb$ for $0$ with parameter values $\epsilon>0$ and $r\geq 0$ and assume that $\xb$ is a tilt-stable local minimizer for $f$ with modulus $\kappa=\sigma'^{-1}$ for some $\sigma'\in(0,\sigma)$ and let $\gamma$, $V$ and $M_\gamma$ be as in Definition \ref{DefTiltStab}.

  Define $g:=(f+\delta_{\oB_\gamma(\xb)})^*$, i.e.,
  \[g(x^*)=\sup\{\skalp{x^*,x}- f(x)\mv x\in \oB_\gamma(\xb)\},\ x^*\in\R^n.\]
  This function $g$ is proper, lsc, convex and finite. Since
  \begin{equation}\label{EqAuxArgMax}\{M_\gamma(x^*)\}=\argmax\{\skalp{x^*,x}- f(x)\mv x\in \oB_\gamma(\xb)\},\ x^*\in V,
  \end{equation}
  one can easily verify that $M_\gamma(x^*)\in\partial g(x^*)$, $x^*\in V$. Hence, for any $x^*,\tilde x^*\in V$ we obtain by the subgradient inequality that
  \begin{align*}\skalp{M_\gamma(x^*),\tilde x^*-x^*}&\leq g(\tilde x^*)-g(x^*)\leq \skalp{M_\gamma(\tilde x^*),\tilde x^*-x^*}\\
  &= \skalp{M_\gamma(x^*),\tilde x^*-x^*}+\skalp{M_\gamma(\tilde x^*)-M_\gamma(x^*),\tilde x^*-x^*}\leq
  \skalp{M_\gamma(x^*),\tilde x^*-x^*}+\kappa\norm{\tilde x^*-x^*}^2.\end{align*}
  implying that $g$ is continuously differentiable on $\inn V$ with $\nabla g(x^*)=M_\gamma(x^*)$.

  Next denote by $\psi$ the function conjugate to $g$, i.e., $\psi(x)=\sup_{x^*}\{\skalp{x,x^*}-g(x^*)\}$ and choose some positive $\epsilon'<\min\{\epsilon, \epsilon/\kappa, 1\}$ with $\B_{\epsilon'}(0)\subset V$. For every $x\in\dom  f\cap\oB_\gamma(\xb)$ and every $x^*\in \R^n$ we have $g(x^*)\geq \skalp{x^*,x}- f(x)$
  implying $\psi(x)\leq \sup_{x^*}\{\skalp{x,x^*}-(\skalp{x^*,x}- f(x))\}= f(x)$ and this inequality trivially holds when $x\not\in\dom f$.

  Consider $x^*\in\B_{\epsilon'}(0)$. Since $(x^*,M_\gamma(x^*))\in\gph\partial g$ we have $(M_\gamma(x^*),x^*)\in\gph\partial\psi$ by \cite[Theorem 23.5]{Ro70} and, since $\psi$ is an lsc convex function, $M_\gamma(x^*)\in \argmin\{\psi-\skalp{x^*,\cdot}\}$ follows. Actually $M_\gamma(x^*)$ is the unique minimizer of $\psi-\skalp{x^*,\cdot}$, since any minimizer $\tilde x$ of $\psi-\skalp{x^*,\cdot}$ satisfies $x^*\in\partial\psi(\tilde x)$ and therefore $\tilde x\in\partial g(x^*)=\{M_\gamma(x^*)\}$ by \cite[Theorem 23.5]{Ro70}. Thus we conclude that $\xb$ is a tilt-stable minimizer for $\psi$ with modulus $\kappa$ and $\tilt(\psi;\xb)=\tilt(f;\xb)$.

   Consider a point $(x,x^*)\in\gph \partial\psi\cap(\B_\gamma(\xb)\times\B_{\epsilon'}(0))$. Then $(x^*,x)\in \gph\partial g\cap(\B_{\epsilon'}(0)\times \B_\gamma(\xb))$ and $g(x^*)+\psi(x)=\skalp{x^*,x}$ by \cite[Theorem 23.5]{Ro70} implying $x=\nabla g(x^*)=M_\gamma(x^*)$ and $\psi(x)=\skalp{x^*,x}-(\skalp{x^*,M_\gamma(x^*)}- f(M_\gamma(x^*)))= f(x)$. Further, by \eqref{EqAuxArgMax} together with $x=M_\gamma(x^*)\in\B_\gamma(\xb)$ the first-order optimality condition
  \[0\in \partial  f(M_\gamma(x^*))-x^*=\partial f(x)-x^*\]
  follows.
  Our choice of $\epsilon'$  ensures that
  \[\norm{x-\xb}=\norm{M_\gamma(x^*)-M_\gamma(0)}\leq\kappa\norm{x^*}<\kappa\epsilon'<\epsilon\]
  and, since $f(x)-\skalp{x^*,x}= f(M_\gamma(x^*))-\skalp{x^*,M_\gamma(x^*)}\leq f(\xb)-\skalp{x^*,\xb}$, we obtain that
  \[ f(x)\leq  f(\xb)+\skalp{x^*,x-\xb}\leq  f(\xb)+\norm{x^*}\norm{x-\xb}<f(\xb)+\kappa{\epsilon'}^2<f(\xb)+\epsilon.\]
  Hence, $\gph \partial\psi\cap(\B_\gamma(\xb)\times \B_{\epsilon'}(0))$ is contained in the graph of the $f$-attentive $\epsilon$-localization $\T_\epsilon^ f$ of $\partial f$ around $(\xb,0)$. Next choose some open convex neighborhood $\tilde U\times\tilde V$ of $(\xb, r\xb)$ such that $\tilde U\subset\B_\gamma(\xb)$ and $\tilde V-r\tilde U\subset \B_{\epsilon'}(0)$. Then for every $(x,x^*)\in \gph(r I+\partial \psi)\cap(\tilde U\times\tilde V)$ we have $(x,x^*-r x)\in \gph \partial \psi\cap \B_\gamma(\xb)\times \B_{\epsilon'}(0)$ and $(x,x^*)\in \gph (r I+\Ta)$ follows.

  Since $\psi$ is a lsc convex function, $\partial\psi$ is a maximally monotone mapping and so is $r I + \partial\psi$ as well. Therefore, the mapping $S:\R^n\tto\R^n$ given by $\gph S=\gph( r  I+\partial \psi)\cap (\tilde U\times\tilde V)$ is locally maximal monotone around $(\xb,\rho\xb)$ by Proposition \ref{PropLocMaxMon}. By the proof of \cite[Theorem 13.36]{RoWe98}, the mapping $r I+ \Ta$ is monotone and, since $\gph S\subset\gph(r I+\Ta)$, there exist neighborhoods $\hat U$ of $\xb$ and $\hat V$ of $r\xb$ such that
  \begin{align*}\gph S\cap(\hat U\times\hat V)&=\gph(r I+\partial\psi)\cap\big((\tilde U\cap \hat U)\times (\tilde V\cap \hat V)\big)\\
  &=\gph(r I+\Ta)\cap (\hat U\times \hat V)=\gph(r I+\Ta)\cap \big((\tilde U\cap \hat U)\times (\tilde V\cap \hat V)\big).
  \end{align*}
  Next we choose some   neighborhoods $\bar U$ of $\xb$ and $\bar V$ of $0$ satisfying $\bar U\subset \tilde U\cap \hat U$ and $r\bar U+\bar V\subset \tilde V\cap\hat V$. It follows that $\gph \partial\psi\cap(\bar U\times\bar V)=\gph\Ta\cap(\bar U\cap\bar V)$ and thus we obtain
  $D^*(\partial\psi)(\xb,0)=D^*\Ta(\xb,0)=D^*_f(\partial f)(\xb,0)$, where the last equality follows from  Proposition \ref{PropEpsLoc}. The lsc convex function $\psi$ is prox-regular and subdifferential continuous at $\xb$ for $0$ and we can use Theorem \ref{ThTiltProxRegSubdiffCont} to obtain
  \[\kappa=\frac1{\sigma'}\geq \sup\left\{\frac{\norm{z}^2}{\skalp{z^*,z}}\bmv (z,z^*)\in \gph D^*(\partial\psi)(\xb,0) = \gph D^*_f(\partial f)(\xb,0)\right\}\mbox{with the convention $\frac 00:=0$.}\]
  Hence
  \[\skalp{z^*,z}\geq \sigma'\norm{z}^2\quad\mbox{whenever}\quad z^*\in D^*_f(\partial f)(\xb,0)(z).\]
  and, since we can choose $\sigma'$ arbitrarily close to $\sigma$, the bound \eqref{EqCharStrVarConv1} follows.

  For later use note that we also have $\Sp(\partial\psi)(\xb,0)=\Sp\Ta(\xb,0)=\Sp_f(\partial f)(\xb,0)$ and consequently
  \begin{equation}\label{EqAuxTilt1}\tilt(\psi;\xb)=\tilt(f;\xb)= \sup\{\norm{PW^{-1}}\mv (P,W)\in \M_{P,W}\partial\psi(\xb,0)=\M_{P,W}^f\partial f(\xb,0)\}\end{equation}
  by Theorem \ref{ThTiltProxRegSubdiffCont}.

  ad ''(iii)$\Rightarrow$(iv)'': Consider $(P,W)\in\M_{P,W}^f\partial f(\xb,0)$ and set $L:=\rge(P,W)$. By taking into account Proposition \ref{PropEpsLoc} we have
  $L=L^*\in \Sp^*\Ta(\xb,0)$ and $L^*\subset\gph D^*\Ta(\xb,0)= \gph D^*_f(\partial f)(\xb,0)$, where $\epsilon>0$ is the parameter from prox-regularity of $f$ at $\xb$ for $0$. Thus, for every $p\in\R^n$ we have $(Pp,Wp)\in \gph D^*_f(\partial f)(\xb,0)$ and from \eqref{EqCharStrVarConv1} we conclude
  \[\skalp{Wp,Pp}=\skalp{PWp,p}=\skalp{PWPp,p}\geq\sigma\norm{Pp}^2.\]

  ad ''(iv)$\Rightarrow$(v)'': By taking into account \eqref{EqPW3}, for every $(P,W)\in \M^f_{P,W}\partial f(\xb,0)$ we have
  \[\skalp{p,(W-\sigma P)p}=\skalp{p, PWPp+(I-P)p}-\sigma\norm{Pp}^2=(\skalp{p,PWPp}-\sigma\norm{Pp}^2)+\norm{(I-P)p}^2,\ p\in\R^n\]
  and the assertion follows.

  ad ''(v)$\Rightarrow$(i)'': Let $r\geq0$ and $\epsilon>0$ denote the parameter values for prox-regularity of $f$ at $\xb$ for $0$. By Remark \ref{RemVarConv}, $f$ is variationally $(-r)$-convex and let $\widehat f,U,V$ and $\rho$ be as Definition \ref{DefVarConv}. By possibly  lowering $\epsilon$ we may assume that $f(\xb)+\epsilon\leq\rho$ and then, by possibly shrinking $U\times V$ and taking into account subdifferential continuity of $\widehat f$ at $\xb$ for $\xba$, we may assume that  $U\times V\subset \B_\epsilon(\xb)\times\B_\epsilon(0)$ and $\widehat f(x)<f(\xb)+\epsilon$, $(x,x^*)\in\gph\partial\widehat f\cap (U\times V)$ implying $\gph \partial\widehat f\cap (U\times V)=\gph \partial_\rho f\cap(U\times V)=\gph \Ta \cap (U\times V)$. Hence $\Sp (\partial \widehat f)(\xb,0)=\Sp_f(\partial f)(\xb,0)$ by Proposition \ref{PropEpsLoc} implying $\M_{P,W}\partial\widehat f(\xb,0)=\M_{P,W}^f\partial f(\xb,0)$.

  Assume now that (v) holds, pick $\sigma'\in(0,\sigma)$ and consider the function $\tilde f:=\widehat f-\frac{\sigma'}2\norm{\cdot-\xb}^2$. Since $\widehat f$ is prox-regular and subdifferentially continuous at $\xb$ for $0$, $\tilde f$ has these properties as well. By \cite[Proposition 3.15]{GfrOut22}  we have
  \begin{align*}\Sp(\partial\tilde f)(\xb,0)&=\{\rge(P,W-\sigma'P)\mv (P,W)\in \M_{P,W}\partial\widehat f(\xb,0)\}\\
  &=\{\rge(P,W-\sigma'P)\mv (P,W)\in \M^f_{P,W}\partial f(\xb,0)\}\end{align*}
  implying $\M_{P,W}\partial\tilde f(\xb,0)=\{(P,W-\sigma'P)\mv (P,W)\in \M_{P,W}^f\partial f(\xb,0)\}$.
  Consider $(P,\tilde W)\in \M_{P,W}\partial\tilde f(\xb,0)$ and the corresponding $(P,W)\in \M^f_{P,W}\partial f(\xb,0)$ with $\tilde W=W-\sigma' P$. Then for every $p\in\R^n$ we obtain that
  \begin{align*}\skalp{p,\tilde Wp}&=\skalp{p,P\tilde WPp+(I-P)p}=\skalp{p,P(W-\sigma P)Pp}+(\sigma-\sigma')\norm{Pp}^2+\norm{(I-P)p}^2\\
  &\geq\min\{\sigma-\sigma',1\}\norm{p}^2
  \end{align*}
  verifying that $\tilde W$ is positive definite and therefore nonsingular. Further, from \eqref{EqPW2} we obtain that
  \[P\tilde W^{-1}=\tilde W^{-1}P=\tilde W^{-1}P\tilde WP\tilde W^{-1}=(P\tilde W^{-1})^T\tilde W(P\tilde W^{-1})\]
  is positive semidefinite and we conclude from Theorem \ref{ThTiltProxRegSubdiffCont} that  $\xb$ is a tilt-stable local minimizer for $\tilde f$. Hence, by Theorem \ref{ThEquivTilt_VSC}, $\tilde f$   is variationally strongly convex at $\xb$ for $0$ and therefore $\widehat f$ is variationally $\sigma'$-convex at $\xb$ for $0$ by Lemma \ref{LemQuadrShift}. Since $f$ has the same property by  Lemma \ref{LemLowerFuncVarConv}, the implication (v)$\Rightarrow$(i) is verified.

  Finally, the equality $\varco(f;\xb|\xba)^{-1}=\tilt(f-\skalp{\xba,\cdot};\xb)$ is a consequence of the equivalence (i)$\Leftrightarrow$(ii). Further the formulas \eqref{EqTilt1}, \eqref{EqTilt2} for the exact bound of tilt-stability follow from the equivalences (ii)$\Leftrightarrow$(iii)$\Leftrightarrow$(iv) and \eqref{EqTilt3} was already deduced in \eqref{EqAuxTilt1}
\end{proof}

We are now in the position to extend the equivalences of Theorem \ref{ThFundEquiv} for variationally strongly convex functions.
\begin{theorem}\label{ThStrMetrRegTiltStab}
  Let $f:\R^n\to\oR$ be an lsc function and let $\xb\in\dom f$ and $\xba\in\partial f(\xb)$ together with the real number $\sigma>0$ be given. Then the following statements are equivalent:
  \begin{enumerate}
        \item[(i)]  $f$ is  variationally strongly convex  at $\xb$ for $\xba$ with modulus $\sigma$.
    \item[(ii)] $f$ is prox-regular at $\xb$ for $\xba$ and  $\xb$ is a tilt-stable local minimizer for $f-\skalp{\xba,\cdot}$ with modulus $\sigma^{-1}$.
    \item[(iii)] $\xb$ is a local minimizer for $f-\skalp{\xba,\cdot}$ and there is some $\rho>f(\xb)$ such that the $f$-truncation $\partial_\rho f$ is strongly metrically regular at $\xb$ for $\xba$ with modulus $\sigma^{-1}$.
  \end{enumerate}
\end{theorem}
\begin{proof}In order to unburden notation, we may again assume that $\xba=0$. The implication (i)$\Rightarrow$(ii) follows from \cite[Theorem 3]{Ro19} together with the fact that variational convexity implies prox-regularity, cf. Remark \ref{RemVarConv}(i). In order to show the  reverse implication, consider the lsc convex function $\psi$ from the proof of the implication (ii)$\Rightarrow$(iii) in Theorem \ref{ThCharStrVarConv} with $\sigma'$ replaced by $\sigma$. Then $\psi$ is prox-regular and subdifferentially continuous, and since $\xb$ is a tilt-stable minimizer for $\psi$ with modulus $\sigma^{-1}$, $\psi$ is variationally strongly convex at $\xb$ for $\xba$ with modulus $\sigma$ by Theorem \ref{ThEquivTilt_VSC} and the same property follows for $f$ by Lemma \ref{LemLowerFuncVarConv}.

For showing the implication (ii)$\Rightarrow$(iii), we  again resort to  the proof of  the implication (ii)$\Rightarrow$(iii) in Theorem \ref{ThCharStrVarConv}. From tilt-stability of $\xb$ for $\psi$ with modulus $\sigma^{-1}$ it follows from Theorem \ref{ThTiltProxRegSubdiffCont} that $\psi$ is metrically strongly regular at $\xb$ with the same modulus. Since we have also shown that $\gph \partial \psi$ and $\gph \Ta$ coincide locally around $(\xb,0)$, $\Ta$ is also strongly metrically regular at $\xb$ for $0$ with modulus $\sigma^{-1}$. By definition, $\Ta$ is a localization of the $f$-attentive truncation $\partial_{f(\xb)+\epsilon} f$ and thus statement (iii) follows with $\rho=f(\xb)+\epsilon$.

Finally let us show the reverse direction (iii)$\Rightarrow$(ii). Assume that (iii) holds, let $f$ be prox-regular with parameters $r\geq 0$ and $\epsilon>0$ and consider a real $\rho>f(\xb)$, open neighborhoods $U$ of $\xb$ and $V$ of $0$ together with a mapping $\vartheta:V\to U$ which is Lipschitz continuous with constant $\sigma^{-1}$, such that $\gph (\partial_\rho f)^{-1}\cap (V\times U)=\gph \vartheta$. We claim that $\xb$ is a strict local minimizer of $f$, i.e., there exists some $\gamma>0$ such that $f(x)>f(\xb)$ holds for all $x\in\oB_\gamma(\xb)\setminus\{\xb\}$. Assuming on the contrary that this claim does not hold, we can find a sequence $x_k\to \xb$, $x_k\not=\xb$, satisfying $f(x_k)\leq f(\xb)$. Since $\xb$ is a local minimizer of $f$, we can conclude that $f(x_k)=f(\xb)$ for all $k$ sufficiently large and consequently $x_k$ is a local minimizer for $f$ as well. Hence, $0\in\partial f(x_k)=\partial_\rho f(x_k)$ contradicting our assumption that $(\partial_\rho f)^{-1}$ has a single-valued localization around $(0,\xb)$. Hence, our claim holds true and, since $f$ is assumed to be lsc, there is some positive real $\eta$ such that $f(x)\geq f(\xb)+\eta$ for all $x$ with $\norm{x-\xb}=\gamma$. Thus, whenever $\norm{x^*}<\bar\delta:=\min\{\eta,\rho-f(\xb)\}/\gamma$ we have
\[M_\gamma(x^*):=\argmin\{f(x)-\skalp{x^*,x}\mv x\in\oB_\gamma(\xb)\}=\argmin\{f(x)-\skalp{x^*,x-\xb}\mv x\in\oB_\gamma(\xb)\}\subset\B_\gamma(\xb)\]
implying $0\in\partial f(x)-x^*$, $x\in M_\gamma(x^*)$, and
\[f(x)-(\rho-f(\xb))<f(x)-\skalp{x^*,x-\xb}\leq f(\xb),\ x\in M_\gamma(x^*).\]
Therefore, for every $x^*\in \B_{\bar \delta}(0)$ and every $x\in M_\gamma(x^*)$ we have $x^*\in \partial_\rho f(x)$ and lower semicontinuity of $f$  ensures that $M_\gamma(x^*)$ is nonempty. By utilizing \cite[Theorem 1.17]{RoWe98}, we can find some open neighborhood $\tilde V\subset \B_{\bar\delta}(0)$ of $0$ such that $M_\gamma(x^*)\subset U$, $x^*\in\tilde V$. It follows that for every $x^*\in V\cap\tilde V$ and every $x\in M_\gamma(x^*)$ we have $(x,x^*)\in\gph \partial_\rho f\cap(U\times V)$ implying $x=\vartheta(x^*)$. Hence $M_\gamma$ coincides with $\vartheta$ on $V\cap\tilde V$ and is therefore  single-valued and Lipschitz continuous with constant $\sigma^{-1}$ on $V\cap\tilde V$ verifying that $\xb$ is a tilt-stable local minimizer for $f$ with modulus $\sigma^{-1}$. There remains to show that $f$ is prox-regular at $\xb$ for $0$. Choose $0<\epsilon<\min\{\gamma,\rho-f(\xb)\}$ such that $\B_\epsilon(\xb)\subset U$ and $\B_\epsilon(0)\subset \tilde V\cap V$. Then for every pair $(x,x^*)\in \gph\partial_{f(\xb)+\epsilon} f\cap(\B_\epsilon(\xb)\times\B_\epsilon(0))$ we have $x=\vartheta(x^*)=M_\gamma(x^*)$ implying that for every $x'\in\B_\epsilon(\xb)\subset \B_\gamma(\xb)$ there holds $f(x')-\skalp{x^*,x'}\geq f(x)-\skalp{x^*,x}$.  From this inequality we easily deduce that $f$ is prox-regular at $\xb$ for $0$ with parameters $r=0$ and the chosen $\epsilon$.
\end{proof}
\begin{remark}
  Though Theorems \ref{ThFundEquiv}, \ref{ThCharStrVarConv} and \ref{ThStrMetrRegTiltStab} provide a bunch of characterizations for tilt-stability and variational strong convexity for prox-regular functions, the list of equivalences can still be extended. E.g., by utilizing the relation (i)$\Leftrightarrow$(ii) from Theorem \ref{ThFundEquiv},  one can easily add to Theorem \ref{ThStrMetrRegTiltStab} the following statement generalizing Theorem 3.2 of Mordukhovich and Nghia \cite{MoNg15} (see also \cite[Theorem 3.3 (3.)]{DruLew13} in the paper of Drusvyatskiy and Lewis):
  \begin{enumerate}
    \item[(iv)]There are neighborhoods $U$ of $\xb$ and $V$ of $\xba$ together with a real $\rho>f(\xb)$ such that the mapping $(\partial_\rho f)^{-1}$ admits a single-valued localization $\vartheta : V \to U$ around $(\xba,\xb)$ and that for any pair $(x^*,x)\in \gph\vartheta=\gph (\partial_\rho f)^{-1}\cap(V\times U)$ we have the uniform second-order growth condition
\[f(x')\geq f(x) +\skalp{x^*,x'-x}+\frac\sigma2\norm{x'-x}^2\]
 whenever $x'\in U$.
  \end{enumerate}
  I am very grateful to one of the reviewers which asked for a comparison of Theorem 5.1 with \cite[Theorem 3.2]{MoNg15} of Mordukhovich and Nghia.. It motivated me to add statement (iii) to the equivalences in Theorem 5.2.
\end{remark}
The statements of Theorem \ref{ThCharStrVarConv}, with the exception of tilt-stability, are also valid for non-positive $\sigma$.
\begin{corollary}\label{CorCharVarConv}
  Let $f:\R^n\to\oR$ be an lsc function and let $\xb\in\dom f$ and $\xba\in\partial f(\xb)$ together with the real $s$ be given. Then the following statements are equivalent:
  \begin{enumerate}
    \item[(i)]  For every $s'\in(-\infty,s)$, $f$ is  variationally $s'$-convex  at $\xb$ for $\xba$.
    \item[(ii)] $f$ is prox-regular at $\xb$ for $\xba$ and
        \begin{equation*}
          \skalp{z^*,z}\geq s\norm{z}^2\quad\mbox{ whenever }\quad z^*\in D_f^*(\partial f)(\xb,\xba)(z),\ z\in\R^n.
        \end{equation*}
        \item[(iii)] $f$ is prox-regular at $\xb$ for $\xba$ and
        \begin{equation*} \skalp{p,PWPp}\geq s\norm{Pp}^2\quad\mbox{ whenever }\quad (P,W)\in\M^f_{P,W}\partial f(\xb,\xba),\ p\in\R^n.
        \end{equation*}
    \item[(iv)] $f$ is prox-regular at $\xb$ for $\xba$ and for every pair $(P,W)\in\M^f_{P,W}\partial f(\xb,\xba)$ the matrix $W-s P$ is positive semidefinite.
  \end{enumerate}
  Finally,  we have the following formulas for the exact bound of  variational convexity:
  \begin{align*}
    \varco(f;\xb|\xba)&=\inf\left\{\frac{\skalp{z^*,z}}{\norm{z}^2}\mv z^*\in D^*_f(\partial f)(\xb,\xba)(z),\ z\in\R^n\right\}\\
    &=\inf\left\{\frac{\skalp{p,PWPp}}{\norm{Pp}^2}\mv (P,W)\in M_{P,W}^f\partial f(\xb,\xba), p\in\R^n\right\}\\
  \end{align*}
  with the convention $\frac 00:=\infty$.
\end{corollary}
\begin{proof}
  Choose $t>0$ such that $\sigma:=s+t>0$ and apply Theorem \ref{ThCharStrVarConv} to the function $\tilde f:=f+\frac t2\norm{\cdot-\xb}^2$. Since we can take $t$ arbitrarily large, the assertions follow from Proposition \ref{PropCalc} and Lemma \ref{LemQuadrShift}.
\end{proof}

\section*{Declarations}

{\bf Competing interests.} The author has no competing interests to declare that are relevant to the content
of this article.\\
{\bf Data availability. }Data sharing not applicable to this article as no datasets were generated or analysed
during the current study.

\end{document}